\documentclass[12pt]{article}
\usepackage{fancyhdr} % gives control over headers/footers.
\usepackage{setspace}
\usepackage{titlesec}
\usepackage{enumitem} % makes our lists nice.
\usepackage{amsmath}
\usepackage{amsthm}
\usepackage{amssymb}
\usepackage{multicol}
\usepackage{pifont,mathrsfs,geometry}
\usepackage{tikz}
\usepackage{tikz-cd}
\usepackage{mathtools}
\usepackage[numbers]{natbib}
\usepackage{float}
\usepackage{pgfplots}
\usepackage{geometry}
\usepackage{xcolor}
\usepackage{hyperref}
\hypersetup{
    colorlinks=true,
    linkcolor=red,
    citecolor=blue,
    urlcolor=blue
}
\usepackage{authblk}

\pgfplotsset{compat=newest}
\usepgfplotslibrary{fillbetween}
\allowdisplaybreaks

\newtheorem{theorem}{Theorem}[section]
\newtheorem{corollary}[theorem]{Corollary}
\newtheorem{lemma}[theorem]{Lemma}
\newtheorem{proposition}[theorem]{Proposition}
\theoremstyle{definition}
\newtheorem{definition}[theorem]{Definition}
\newtheorem{notation}[theorem]{Notation}
\newtheorem{remark}[theorem]{Remark}
\newtheorem{example}[theorem]{Example}

\def\subsection{\def\@secnumfont{\bfseries}\@startsection{subsection}{1}%
  {\parindent}{.5\linespacing\@plus.7\linespacing}{-.5em}%
  {\normalfont\bfseries}}
\def\pts#1{\hfill({\em #1 point\ifnum #1>1 s\fi})\ifnum#1>9\else\phantom{\em 1{\ifnum#1>1\else s\fi}}\fi}

\makeatother
 \titleformat{\section}
  {\normalfont\scshape \filcenter}{\thesection}{.5em}{}

\newcommand{\B}{\beta}
\newcommand{\p}{\partial}
\newcommand{\K}{\kappa}

\newcommand{\hd}{\hat{d}}
\newcommand{\hX}{\widehat{X}}

\begin{document}

\newcommand{\ob}{\mathfrak{o}}
\newcommand{\TitleString}{CHARACTERIZATION OF SUBLINEARLY MORSE GEODESICS}

% \pagestyle{myheadings}
% \markright{\TitleString}
\thispagestyle{empty}

 \begin{center}
     \textbf{CURTAIN CHARACTERIZATION OF SUBLINEARLY MORSE GEODESICS IN CAT(0) SPACES}
 \end{center}
 
 \begin{center}
     Elliott Vest
 \end{center}

\vspace{.5cm}
 \begin{abstract}
 We show that the sublinear Morse boundary of every CAT(0) space continuously injects into the Gromov
boundary of a hyperbolic space, which was not previously known even for all CAT(0) cube complexes.
Our work utilizes the curtain machinery introduced by Petyt-Spriano-Zalloum. Curtains are more general
combinatorial analogues of hyperplanes in cube complexes, and we develop multiple curtain characterizations of the sublinear Morse property along the way. Our results answer multiple questions of Petyt-Spriano-Zalloum.

% As a Corollary, under the correct finite moment assumption on the probability measure, we obtain that the curtain model of a CAT(0) group becomes a model for its Poisson Boundary recovering a result from 
 
 \end{abstract}
 
 \section{Introduction}

The \textit{sublinearly Morse boundary}, denoted $\p_\K X$, was introduced by Qing-Rafi-Tiozzo in \cite{QRT19,QRT20} as an extension of the Morse boundary \cite{Cor16}. It remains a quasi-isometry invariant while also capturing the asymptotic behavior of random walks on finitely generated groups. A motivating application they use is showing that, for each finite-type surface $S$ and some $p \in \mathbb{N}$ depending on $S$, the $\log^p$-Morse boundary for the mapping class group of $S$ serves as a topological model for the Poisson boundary for its random walks. This boundary has now subsequently been studied in the CAT(0) cube complex setting \cite{MQZ20,IMZ21}, and more generally the hierarchically hyperbolic setting \cite{DZ22, NQ22}.

When in regards to CAT(0) spaces, authors Petyt-Spriano-Zalloum in \cite{PSZ22} introduce a combinatorial tool named a \textit{curtain} that serves as an analogue to a hyperplane for a CAT(0) cube complex. Building off of ``hyperplane-separation" metrics introduced by Genevois \cite{Gen20b}, the authors utilize curtains in a CAT(0) space $X$ to build the \textit{curtain model} - a hyperbolic space whose isometry group contains $\text{Isom}\hspace{.1cm} X$. Denoting $\widehat{X}$ as the curtain model and $\p\widehat X$ as its Gromov boundary, we summarize our main results. \vspace{.25cm}

\noindent \textbf{Main Results:} \textit{Let $X$ be a proper CAT(0) space and $\K$ be a sublinear function. We show the following, }

\begin{enumerate}

    \item \textit{If $\K^4$ is sublinear and $\p_{\K} X$ is endowed with the sublinearly Morse topology, then $\p_{\K} X$ continuously injects into the Gromov boundary $\p \widehat{X}$. Endowed with the cone topology, $\p_{\K} X$ topologically embeds into $\p \widehat{X}$. (Theorem \hyperlink{TheoremA}{A}.)}
    \item \textit{Any $\K$-Morse ray can be characterized by a dual chain of $\K$-separated curtains crossing the ray at a sublinear rate. (Theorem \hyperlink{TheoremC}{C}.)}
    \item \textit{Any $\K$-Morse ray can be characterized by the $\K$-persistence of its projection into $\widehat{X}$. (Theorem \hyperlink{TheoremD}{D}.)}
    
\end{enumerate}

Theorem \hyperlink{TheoremA}{A} upgrades a similar theorem by Petyt-Spriano-Zallouom in \cite{PSZ22} from the Morse setting. Since sublinearly Morse rays have been shown to be unparametrized quasi-geodesics in the hyperbolic space $\widehat{X}$, a key ingredient we prove is that these rays are also unbounded in $\widehat{X}$ and, thus, define a point in $\p \widehat{X}$. Moreover, as random walks sublinearly track $\K$-Morse rays in CAT(0) spaces \cite{Choi}, Theorem \hyperlink{TheoremA}{A} immediately concludes that the Gromov boundary of the curtain model for a CAT(0) group is a model for its Poisson boundary (Corollary \hyperlink{CorollaryB}{B}).  The central characterization used to give Theorem \hyperlink{TheoremA}{A} is Theorem \hyperlink{TheoremC}{C} - a CAT(0) analogue of a result in \cite{MQZ20} in the cube complex setting that allows directions of $\K$-Morse rays to be described in a combinatorial fashion. Lastly, a version of Theorem \hyperlink{TheoremD}{D} was proven for hierarchically hyperbolic spaces in \cite{DZ22}, including mapping class groups of finite-type surfaces.  In particular, they proved that $\kappa$-Morse rays in the mapping class group make persistent and fast (compared to $\kappa$) progress in the curve graph of the whole surface. Hence, Theorem \hyperlink{TheoremD}{D} further strengthens observations that the curtain model for a CAT(0) space gives a similar role of the curve graph for a mapping class group of a finite-type surface \cite{Zal23}. Further motivation, details, and the full statements of our theorems will be given in the rest of the introduction.

\vspace{.25cm}

\noindent 1.1. \textbf{Hyperbolic-like boundaries and cube complexes summary.} In recent years, there have been numerous advancements in investigating the boundaries of spaces, either to gain insights into the spaces themselves or to understand groups that act geometrically on these spaces. A foundational result is due to Gromov \cite{Gro87}, who demonstrated that quasi-isometries between hyperbolic spaces induce homeomorphisms on their visual boundaries, leading to a well-defined notion of a boundary of a hyperbolic group. However, Croke and Kleiner showed that this result does not hold for CAT(0) spaces, as they discovered two quasi-isometric CAT(0) spaces that lack homeomorphic visual boundaries \cite{CK00}. To address this problem in the CAT(0) setting, Charney and Sultan introduced the contracting boundary, a quasi-isometry invariant restriction of the visual boundary that only looks at the ``hyperbolic" directions of a space \cite{CS15}. Subsequently, Cordes extended this result to any proper geodesic space by defining Morse geodesics and the Morse boundary \cite{Cor16}. It is worth noting that the contracting and Morse conditions are equivalent in the CAT(0) setting, and if the underlying space is hyperbolic, then all mentioned boundaries are equivalent.

The \textit{sublinear Morse boundary}, denoted $\p_\K X$ where $\K$ is a sublinear function, was introduced in \cite{QRT19,QRT20} with the motivation to preserve a notion of a hyperbolic-like boundary that is quasi-isometry invariant while also capturing generic directions of the space in question. More specifically, random walks can sublinearly track the sublinearly Morse directions of a group $G$ to serve as a topological model for the Poisson boundary when $G$ is a right angled Artin group \cite{QRT19} or a mapping class group of a finite-type surface \cite{QRT20}. This result has also been shown to hold for rank-1 CAT(0) spaces and Teichm\"{u}ller spaces of finite-type surfaces \cite{GQR22}, then later with CAT(0) admissible groups with mild assumptions \cite{NQ22}.

Apart from the aforementioned developments, CAT(0) cube complexes have also been of particular interest due to their combinatorial nature \cite{Sag97}. When one \textit{cubulates} a group, i.e. shows that the group acts geometrically on a CAT(0) cube complex, one can import the various combinatorial information of the cube complex to the group. Notable applications of this strategy are shown in the resolutions of the virtual Haken and fibering conjectures in 3-manifold theory \cite{Ago13,Wis21}. There are many examples of interesting groups that can be cubulated such as right-angled Artin groups \cite{CD95}, Coxeter groups \cite{NR03}, small cancellation groups \cite{Wis04}, hyperbolic 3-manifold groups \cite{KM12}, and others. Furthermore, CAT(0) cube complexes have become a useful tool for studying the geometry of mapping class groups of finite-type surfaces, and hierarchically hyperbolic spaces more generally \cite{BHS21, DMS20, HHP22, DZ22}. 

\vspace{.25cm}

\noindent 1.2. \textbf{The curtain model and a continuous injection.} Recently, Petyt-Spriano-Zalloum introduced an analogue for a hyperplane in any CAT(0) space \cite{PSZ22}. Hyperplanes are the basic combinatorial objects in a cube complex, and their general analogues, \textit{curtains}, have a simple definition: each is the preimage of a unit length interval of a geodesic under closest point projection (Definition \ref{curtain}). Like hyperplanes, curtains separate the ambient space into two components, and Petyt-Spriano-Zalloum prove that an analogous notion of separation allows one to build a hyperbolic space, called \textit{the curtain model} (Definition \ref{Curtain Model}), which encodes much of the ambient hyperbolic geometry of the space. Such a space serves as an analogue of curve graphs for mapping class groups \cite{Zal23}, as the curve graph of a finite-type surface also encodes the hyperbolic geometry of its associated mapping class group \cite{MM99}. In addition, ongoing work of Le Bars shows the utility curtains can have when investigating asymptotic behavior of random walks in CAT(0) spaces \cite{LeB1, LeB2}. This segues to our main result, which deals in the extension of the projection of a CAT(0) space $X \longrightarrow \widehat{X}$ to a continuous injection of its sublinearly Morse boundary\vspace{.25cm}\hypertarget{TheoremA}{.}  
 \\ 
\noindent \textbf{Theorem A.} \textit{Let $X$ be a proper CAT(0) space, $\widehat{X}$ its curtain model, and $\K$ be a sublinear function such that $\K^4$ is sublinear. Endow $\p_\K X$ with the sublinear Morse topology and denote $\p \widehat{X}$ as the Gromov boundary of $\widehat{X}$ . Then the projection map $ X \rightarrow \widehat{X}$ extends to an $\text{Isom}\hspace{.1cm} X$-equivariant continuous injection $\varphi: \p_\K X \lhook\joinrel\longrightarrow \p \widehat{X}$. Moreover, endowing $\p_\K X$ with the subspace topology of the cone topology makes $\varphi$ a homeomorphism onto its image.} \vspace{.25cm}

In other words, the curtain model of a CAT(0) space will capture all the generic directions of the CAT(0) space. We note that Theorem \hyperlink{TheoremA}{A} relies on the condition that $\K$ also have powers that are sublinear. Given that sublinear Morse boundaries are commonly employed in applications where random walks converge to $\p_\K X$ for $\K(t) = \log^p(t)$ \cite{NQ22, QRT20}, the cost of our restriction on $\K$ is justifiable for cleaner arguments. In addition, the assumption that the CAT(0) space be proper is only used for the characterization that $\K$-Morse rays are $\K$-contracting \cite{QRT19}. If one worked with the sublinearly contracting boundary instead of $\partial_\K X$, our work could yield the same result without the assumption that the CAT(0) space be proper.

In the recent paper \cite{CFFT22}, the authors show if a group $G$ acts on a hyperbolic space $\widehat X$ with a WPD element and $\mu$ is a probability measure with finite entropy, then $\partial \widehat X$ with the hitting measure is a model for the Poisson boundary of $(G,\mu)$. In the context of CAT(0) spaces, work of Choi in \cite{Choi, choi22} has shown that  random walks with finite $p$th moment will sublinearly track $o(n^{1/p})$-Morse geodesics. Moreover, \cite{QRT20} shows that whenever random walks sublinearly track $\K$-Morse rays, $\partial_\K X$ with its hitting measure is a model for the Poisson boundary of $(G,\mu)$. Thus, the injection created in Theorem \hyperlink{TheoremA}{A} immediately gives the following corollary. This corollary  recovers the aforementioned result in \cite{CFFT22}, but with a stronger assumption on the probability measure $\mu$\hypertarget{CorollaryB}{.}  \vspace{.25cm}

\noindent \textbf{Corollary B.} \textit{Let $G$ be a group that acts properly on a proper CAT(0) space $X$, denote $\widehat X$ as its curtain model, and let $\mu$ be a non-elementary probability measure on G with finite 5th moment. Let $\nu$ be the hitting measure on the Gromov boundary $\p \widehat X$. Then $(\p \widehat{X}, \nu)$ is a model for the Poisson boundary of $(G,\mu)$}.\vspace{.25cm}

Endowing $\p_\K X$ with the sublinearly Morse topology is nice because it makes $\p_\K X$ metrizable and a quasi-isometry invariant \cite{QRT19}. However, our proof of continuity only uses open sets in the \textit{cone topology} which is strictly coarser than the sublinearly Morse topology \cite{IMZ21}. We use techniques of \cite{AM22} to show that $\varphi$ can be a homeomorphism on its image when $\p_\K X$ is endowed with the cone topology (this is Theorem \ref{homeo} in our paper).

 Theorem \hyperlink{TheoremA}{A} has been shown in the CAT(0) cube complex setting \cite{IMZ21}, where the authors project to a hyperbolic space inspired by Genevois's work \cite{Gen20b}. However, their proof relied on the assumption that the cube complex admits a \textit{factor system} - namely that the cube complex be a hierarchically hyperbolic space  \cite{BHS17}. Recently, Shepherd found an example of a CAT(0) cube complex that does not admit a factor system \cite{She22}. As Theorem \hyperlink{TheoremA}{A} applies to the CAT(0) setting, it in particular applies to any CAT(0) cube complex - including those not covered in prior literature. Similar theorems like Theorem \hyperlink{TheoremA}{A} have been shown for sublinear Morse boundaries of hierarchically hyperbolic spaces \cite{DZ22} and the Morse boundary of CAT(0) spaces \cite{PSZ22}, both of such results requiring a leverage of cubical techniques. Since He in \cite{He23} has shown the Morse boundary with the Cashen-Mackay topology to be homeomorphic to $\p_\K X$ for $\K\equiv 1$, Theorem \hyperlink{TheoremA}{A} recovers the Morse boundary result in \cite{PSZ22} when in regards to the Cashen-Mackay topology. Additional related work of Theorem \hyperlink{TheoremA}{A} can be found in \cite{AM22},  where Abbott and Incerti-Medici study classes of spaces that have \textit{hyperbolic projections} and  are \textit{$\K$-injective}. Hence, Theorem \hyperlink{TheoremA}{A} shows that rank-one CAT(0) spaces and groups fall into the class of objects that Abbott and Incerti-Medici build a framework for studying.
 
 % Additionally, \cite{AM22} explores the necessary conditions to compare a group's actions on two CAT(0) cube complexes as well as when topologies of their their sublinear Morse boundaries coincide with the visual topology.
\vspace{.25cm}
\noindent 1.3. \textbf{Characterizations of $\K$-Morse rays and hyperbolicity.} The key arguments needed to prove our main result involve characterizing $\K$-Morse geodesics in the same combinatorial fashion as done in the CAT(0) cube complex setting  \cite{CS15, MQZ20}\hypertarget{TheoremC}{.} \vspace{.25cm} \\
\noindent \textbf{Theorem C.} \label{TheoremB}   \textit{In a CAT(0) space, a geodesic ray $b$ is $\K$-contracting if and only if there exists $C>0$ such that $b$ meets a chain of curtains $\{h_i\}$ at points $b(t_i) \in h_i$ satisfying:
\begin{itemize}
     \item $t_{i+1}-t_i\leq C\K(t_{i+1})$
     \item $h_i$ and $h_{i+1}$ are $C\K(t_{i+1}$)-separated 
 \end{itemize}}

A geodesic ray with a dual chain of curtains satisfying the above conditions will be defined as a \textit{$\K$-curtain excursion geodesic} in reference to similarly defined rays in the CAT(0) cube complex setting \cite{MQZ20}. Thus, since the proof of the forward direction finds a dual chain (Proposition \ref{ContractingToExcursion}), Theorem \hyperlink{TheoremC}{C} states that $\K$-contracting rays are $\K$-curtain excursion rays and vice versa. One doesn't need the curtains to be dual to the geodesic to prove the reverse direction, as Proposition \ref{NonDualto Contracting} shows, but a dual chain is usually preferred for more straightforward arguments. In fact, Theorem \hyperlink{TheoremC}{C} indirectly shows that a geodesic meeting a chain as above will also give a  dual chain that meets the geodesic with the same properties as well (see Corollary \ref{dualizing excursion }). Theorem \hyperlink{TheoremC}{C} plays a critical role in the proof of Theorem \hyperlink{TheoremA}{A}  similarly to how the CAT(0) cube complex version of Theorem \hyperlink{TheoremC}{C} in \cite{MQZ20} is used to prove main results in \cite{IMZ21,DZ22,AM22}. Now such a characterization can be applied to all CAT(0) groups.

 Other characterizations of $\K$-contracting rays were shown in \cite{MQZ20} in the CAT(0) cube complex setting, one of which was later defined as a \textit{$\K$-persistent shadow} \cite{DZ22}.  A geodesic ray $b$ in a CAT(0) space $X$ with infinite diameter in the curtain model has a $\K$-persistent shadow if there exists a $C>0$ such that for all $s<t$, 
    $$\hat{d}(b(s),b(t)) \geq C \cdot \displaystyle \frac{t-s}{\K(t)} - C.$$ where $\hat{d}$ is the distance in the curtain model $\widehat{X}$. We also give a similar characterization in the CAT(0) setting\hypertarget{TheoremD}{:}  \vspace{.25cm}

\noindent \textbf{Theorem D.}   \textit{ Let $b$ be geodesic ray emanating from $\ob\in X$ with infinite diameter in the curtain model $\widehat{X}$.\begin{itemize}
     \item If $b$ is $\K$-contracting and $\K^4$ is sublinear, then $b$ has a $\K^4$-persistent shadow in the $\hd$ metric.
     \item If $b$ has a $\K$ persistent shadow in the $\hat{d}$ metric and $\K^2$ is sublinear, then $b$ is $\K^2$-contracting.
 \end{itemize} }
  
Note that when $\K \equiv 1$, Theorem \hyperlink{TheoremD}{D} is equivalent to a geodesic being contracting if and only if the geodesic projects to a \textit{parametrized} quasi-geodesic in $\widehat{X}$. This is a known result in the mapping class group setting when projecting to its curve graph \cite{Ber06, DR10}, and \cite{DZ22} extend the characterization to the sublinear Morse setting in hierarchically hyperbolic spaces. Thus, Theorem \hyperlink{TheoremD}{D} shows yet another comparison that the curve graph is to the mapping class group as the curtain model is to its CAT(0) space (see Section 8 of \cite{Zal23}).

Lastly, Genevois has shown a hyperbolicity criterion for CAT(0) cubes complexes through grids of hyperplanes. As an application of Theorem \hyperlink{TheoremC}{C} in the Morse setting, we create a curtain version for this criterion. Two collections of chains $\mathcal{H}, \mathcal{K}$ are said to be a \textit{curtain grid} if all curtains in $\mathcal{H}$ cross all curtains in $\mathcal{K}$. We prove the following\hypertarget{TheoremE}{.} \vspace{.01cm}

\noindent \textbf{Theorem E.} \textit{ Let $X$ be a CAT(0) space. Then $X$ is hyperbolic if and only if every curtain grid is $E$-thin for some uniform $E > 0$. \vspace{.25cm}}

A curtain grid ($\mathcal{H}, \mathcal{K})$ is \textit{$E$-thin }if $\operatorname{min}\{|\mathcal{H}|, |\mathcal{K}|\} \leq E$. So, similar to the cube complex setting, the intuition behind Theorem \hyperlink{TheoremE}{E} is that large regions of crossing curtain chains imply large regions of flatness in the space. Thus, when one bounds the thickness of all curtain grids in the space, one obtains hyperbolicity.\vspace{.25cm}

\noindent 1.4. \textbf{Further questions.} \vspace{.25cm}

\noindent \textbf{Question 1.} Future work of Petyt, Spriano, and Zalloum will extend curtain machinery to a larger class of spaces than CAT(0) spaces. Zalloum gives a great survey discussing the recent developments of curtains and their relationships with hierarchically hyperbolic spaces in \cite{Zal23}. If the curtain model can be extended to a larger class of spaces, will the results and techniques in this paper also extend as well? In what generality can sublinear Morseness be characterized by curtain excursion?
\vspace{.25cm}

\noindent \textbf{Question 2.} As seen in Definition \ref{Curtain Model}, the distance function in the curtain model is described through a family of $d_L$ metrics inspired by Genevois's hyperplane-separation metrics in \cite{Gen20b}. It is likely that, in the cube complex setting, one could create a hyperbolic ``hyperplane model" that uses hyperplane-separating metrics instead of curtain-separating metrics. If such a hyperplane model can be created, would this model be preferred to the curtain model in the CAT(0) cube complex setting? Would the curtain model be quasi-isometric to the ``hyperplane model"?
\vspace{.25cm}

\noindent \textbf{Outline and proof summaries:} Section \ref{Section 2} covers the necessary background information for all the work in this paper. Section \ref{Section 3} proves Theorem \hyperlink{TheoremC}{C}, where the forward and backward directions are Propositions \ref{ContractingToExcursion} and \ref{NonDualto Contracting}, respectively. Our proofs required synthesizing the curtain machinery along with techniques of \cite{PSZ22,MQZ20} to give a more generalized characterization in the CAT(0) setting. Prior techniques for similar characterizations of Theorem \hyperlink{TheoremC}{C} would give a weaker converse direction in our context, so we use different techniques than prior literature to prove rays \textit{meeting} a $\K$-chain of curtains will be $\K$-contracting. 

Section \ref{Section 4} shows the continuous injection given in Theorem \hyperlink{TheoremA}{A} with Theorem \ref{sublinearly morse continuous} stating continuity with respect to the sublinearly Morse topology and Theorem \ref{Continuous}/\ref{homeo} giving a homeomorphic notion when $\p_\K X$ is endowed with the cone topology. A major contribution of our work is showing $\K$-contracting rays are unbounded in the curtain model (Lemma \ref{unbnd}), and it's why we need the assumption that $\K^4$ is also a sublinear function (as Example \ref{example} shows). This gives that the identity map $X \longrightarrow \widehat X$ can extend to a well defined map  $\varphi: \p_\K X \longrightarrow \p \widehat{X}$. Our arguments for injectivity rely heavily on the characterization stated in Theorem \hyperlink{TheoremC}{C} in order to give proofs of a combinatorial nature. As the cone topology and the \textit{curtain topology} (Definition \ref{curtain top}) are equal on $\p_\K X$ (Lemma \ref{hyp open}), we then use Theorem  \hyperlink{TheoremC}{C} along with open sets defined via collections of rays crossing a curtain to show continuity of $\varphi$. 

Section \ref{Section 5} proves Theorem \hyperlink{TheoremD}{D} (Theorem \ref{persistent characterization}). The work in showing $\K$-contracting rays are unbounded in the curtain model also shows the forward direction of Theorem \hyperlink{TheoremD}{D}, and its argument has a similar flavor to the analogous arguments made in \cite{MQZ20} for the cube complex setting. On the other hand, the reverse direction of Theorem \hyperlink{TheoremD}{D} is entirely original. It is proven by first subdividing the ray into segments of length bounded by $\K^2$ and then deducing there must be at least three $\K$-separated curtains dual to the geodesic in each segment. The first of each set of three curtains creates a $\K^2$-chain dual to the geodesic and Theorem \hyperlink{TheoremC}{C} gives $\K^2$-contracting. Lastly, Section \ref{Section 6} proves Theorem \hyperlink{TheoremE}{E} (Theorem \ref{grid theorem}). The forward and converse assumptions both give that curtains dual to the same geodesic at a sufficiently far distance from each other must be $L$-separated for some uniform $L$. This connection leads to the equivalence in our statement.
\vspace{.25cm}

\noindent \textbf{Acknowledgements:} I would like to thank my advisor Matthew Gentry Durham for his amazing support, guidance, weekly meetings, and feedback on this project. I'd also like to thank Kunal Chawla, Merlin Incerti-Medici, Jacob Garcia, Anthony Genevois, Vivian He, Harry Petyt, Yulan Qing, and Abdul Zalloum for extremely helpful conversations and comments.

More specifically,  I am very appreciative of Abdul for his multiple and valuable discussions about the curtain model in addition to his encouragement for me to pursue many of the problems in this project. I am extremely thankful for Kunal's discussion on walking through the proper arguments for Corollary \hyperlink{CorollaryB}{B}. I am also grateful to Harry for our conversation that inspired the development of Subsection \hyperlink{Section3.2}{3.2}. Lastly, the warmest of thanks goes to Jacob Garcia for the many hours of friendly help and guidance in my PhD program.
 
 \section{ Preliminaries}\label{Section 2}

 This section will cover relevant background for the paper. Subsection \hyperlink{Section2.1}{2.1} recalls CAT(0) spaces and their geometry, Subsection \hyperlink{Section2.2}{2.2} reviews sublinearly Morse boundaries, and Subsection \hyperlink{Section2.3}{2.3} gives key definitions and lemmas in \cite{PSZ22} of the curtain machinery we will using throughout the paper\hypertarget{Section2.1}{.} \\\\
\noindent 2.1. \textbf{CAT(0) spaces.}

 % \subsection{ CAT(0) Spaces.}

 \begin{definition}[Quasi-isometric embedding, Quasi-isometry]
             Let $(X,d_X)$ and $(Y,d_Y)$ be metric spaces. A function $f:X\rightarrow Y$ is called a \textit{$(K,C)$-quasi-isometric embedding} if, for every pair of points $x,x'\in X$ we have
         $$\frac{1}{K}d_X(x,x') -C \leq d_Y(f(x),f(x'))\leq Kd_X(x,x')+C. $$

         If for all $y \in Y$, $d_Y(f(X),y) \leq C$, we call $f$ a \textit{quasi-isometry}.
 \end{definition}

 \begin{definition}[Geodesics, Quasi-geodesics] A \textit{geodesic ray} in a space $X$ is a isometric embedding $b: [0,\infty) \rightarrow X$. As a convention, we fix some $\ob \in X$ and assume always that $b(0) = \ob$. Similarly, a \textit{quasi-geodesic ray} is a quasi-isometric embedding $\B : [0,\infty) \rightarrow X$ based at $\ob$.
     
 \end{definition}

 \begin{definition}[CAT(0) Space]
     A geodesic metric space $(X,d_X)$ is said to be \textit{CAT(0)} if geodesic triangles in $X$ are at least as thin as corresponding representative triangles in Euclidean space. More precisely, given any triangle $\triangle xyz$, one can create a representative triangle in the Euclidean plane $\triangle \overline{xyz}$ where matching sides of both triangles have the same lengths. See Figure \ref{fig:CAT(0)}.      If one picks any point  $p$ on $ \triangle xyz$, say $p$ is on edge $[y,z]$, there exists a corresponding point $\overline p$ on $\triangle \overline{xyz}$ such that $d_X(y,p) = d_{\mathbb{E}^2}(\overline{y}, \overline{p})$ and $d_X(p,z) = d_{\mathbb{E}^2}(\overline{p}, \overline{z})$. Now, picking any two points $p, q$ on $\triangle xyz$, a CAT(0) space must have the relationship: $$d_X(p,q) \leq d_{\mathbb{E}^2}(\overline{p}, \overline{q}).$$

 \end{definition}

        \begin{figure}[ht]
        \centering
        \begin{tikzpicture}[scale=.8]
    
    \draw[fill] (0,3) circle [radius=0.05];
    \draw[fill] (-2,0) circle [radius=0.05];
    \draw[fill] (2,0) circle [radius=0.05];
    \draw (0,3) to[out=250, in=50] (-2,0);
    \draw (-2,0) to[out=20, in=160] (2,0);
     \draw (2,0) to[out=130, in=290] (0,3);
     \node[above] at (0,3.2) {$\triangle xyz$ in $X$};
     \node[right] at (0,3) {$x$};
     \node[below] at (-2,0) {$y$};
     \node[below] at (2,0) {$z$};
     \node[below] at (.6,.35) {$p$};
     \node[above] at (1, 1.2) {$q$};
     \draw[fill] (1,1.2) circle [radius=0.05];
     \draw[fill] (.6,.35) circle [radius=0.05];
     \draw[blue] (.6,.35) to (1,1.2);

    \draw[fill] (7,3) circle [radius=0.05];
    \draw[fill] (5,0) circle [radius=0.05];
    \draw[fill] (9,0) circle [radius=0.05];
    \draw (7,3) to (5,0);
    \draw (5,0) to (9,0);
    \draw (9,0) to (7,3);
    \node[above] at (7,3.2) {$\triangle \overline{xyz}$ in $\mathbb{E}^2$};
    \node[right] at (7,3) {$\bar x$};
    \node[below] at (5,0) {$ \bar y$};
    \node[below] at (9,0) {$\bar z$};
     \draw[fill] (7.5,0) circle [radius=0.05];
     \draw[fill] (8.1,1.35) circle [radius=0.05];
    \node[below] at (7.5,0) {$\bar p$};
    \node[above] at (8.1, 1.35) {$\bar q$};
    \draw[blue] (7.5,0) to (8.1, 1.35);
    %\draw[help lines] (0,0) grid (10,6);
    %\node [below] at (0,0) {$\null$};
    %\node [below] at (10,6) {$\null$};
    \end{tikzpicture}

        \caption{For a CAT(0) space $X$, the geodesic $[x,z]$ in $X$ is the same length as the geodesic $[\overline x, \overline z]$ in $\mathbb{E}^2$. The same relationship holds for the other two sides of both triangles.}
        \label{fig:CAT(0)}
    \end{figure}
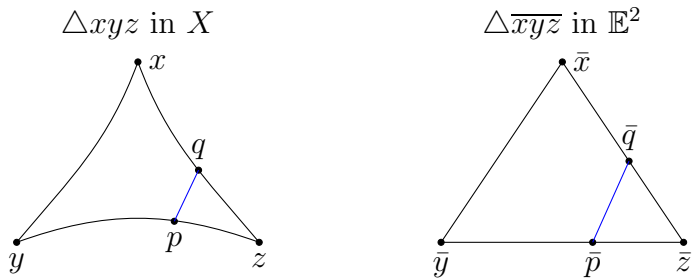

 When working with a CAT(0) space, one gets to enjoy a number of useful properties below.

 \begin{lemma}[See \cite{BH99}] A CAT(0) space $X$ has the following properties:

 \begin{enumerate}
     \item It is uniquely geodesic, that is, for any two points $x, y$ in $X$, there exists exactly one geodesic connecting them.
     \item The nearest-point projection from a point $x$ to a geodesic $b$ is a unique point. In fact, the closest-point projection map$$\pi_b: X \rightarrow b$$ is 1-Lipschitz.
\item Convexity: For any convex set $Z \in X$, the distance function $f: X \rightarrow \mathbb{R}^{+}$ given by $f(x)=d(x, Z)$ is convex in the sense of \cite[II.2.1]{BH99}\hypertarget{Section2.2}{.}

 \end{enumerate}

 \end{lemma}

 \noindent 2.2. \textbf{ Sublinearly-Morse boundaries.}\vspace{.25cm}

 In the context of CAT(0) spaces, geodesic rays based at $\ob \in X$ determine a ``direction tending towards infinity". The collection of all these directions give a notion of a boundary. It is the restriction of this collection to sublinearly Morse geodesic rays which creates a boundary that is both a quasi-isometry invariant and still captures the ``generic" directions of $X$. We review the definition and properties of $\kappa$ Morse geodesic rays needed for this paper. For further details and expansion on below definitions, see \cite{QRT19}.
 
\begin{definition}[Sublinear function]
    We fix a function

    \begin{center}
        $\kappa:[0, \infty) \rightarrow[1, \infty) \quad$ such that $\quad \displaystyle \lim _{t \rightarrow \infty} \frac{\kappa(t)}{t}=0. $
    \end{center}

This second requirement makes $\K$ \textit{sublinear}. We also add the convention that $\K$ be monotone increasing and concave (This is not a necessary convention, but it makes arguments cleaner - see \cite{QRT19}). For the chosen fixed basepoint $\ob \in X$, we denote 

$$ \K(x) = \K(d(\ob,x)).$$

\end{definition}

Oftentimes, we create upper bounds of distances in terms of $\K$ with a specific input. The following computational lemma gives us a way to change our inputs of $\K$ for a more suitable upper bound, and it is used throughout the paper.

\begin{lemma}[Lemma 3.2 in \cite{QRT19}]
\label{3.2}For any $D_0>0$, there exist $D_1,D_2>0$ depending only on $D_0$ and $\K$ so that for $x,y\in X$, we have \begin{center}
     $d(x,y)\leq D_0\K(x) \Rightarrow D_1\K(x)\leq\K(y)\leq D_2\K(x).$
 \end{center}

\end{lemma}

\begin{definition}[$\kappa$-contracting geodesic, contracting constant]
     A geodesic ray $b$ is said to be\textit{ $\kappa$-contracting} if there exists a constant $C \geq 0$ such that for any $x \in X$ and any ball $B$ centered at $x$ with $B \cap b=\emptyset$, we have $\operatorname{diam}\left(\pi_{b}(B)\right)<C \cdot \kappa(x)$. We call the constant $C$ the \textit{contracting constant.}
\end{definition}

\begin{definition}[$\kappa$-neighborhood]
    For a closed set $Z$ and a constant $\mathrm{n} \geq 0$, define the \textit{$(\kappa, \mathrm{n})$-neighbourhood} of $Z$ to be

$$
\mathcal{N}_{\kappa}(Z, \mathrm{n})=\left\{x \in X \mid d_{X}(x, Z) \leq \mathrm{n} \cdot \kappa(x)\right\} .
$$
Note: We often abbreviate to \textit{$\K$-neighborhood}, and our closed set $Z$ will always be a geodesic or quasi-geodesic in this paper.
\end{definition}

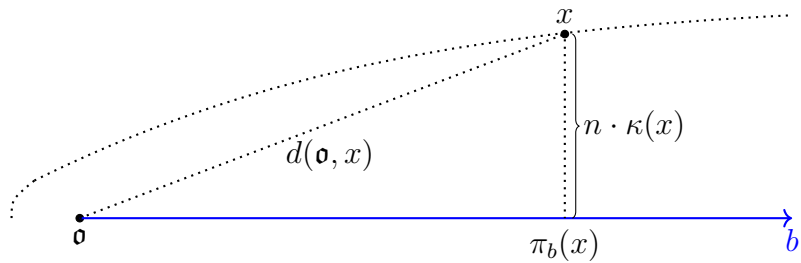
\begin{figure}[ht]
    \centering

    \begin{tikzpicture}[scale=1]
    %\draw[help lines] (0,0) grid (10,6);
    %\node [below] at (0,0) {$\null$};
    %\node [below] at (10,6) {$\null$};
    \node[below] at (.6,-2) {$\mathfrak{o}$};
    \draw[fill] (.6,-2) circle [radius=0.05];
    \draw [blue,thick, ->](.6,-2) to (10,-2);
    \node [below, blue] at (10,-2) {$b$};
    \draw[dotted, thick] (7,.38) to (7,-2);
    \draw[fill] (7,.45) circle [radius=0.05];
    \node[above] at (7,.45) { $x$};
    \draw[dotted, thick] (.6,-2) to (7,.45);
    \node[below] at (3.9, -.8) { $d(\ob,x)$};
    \node[below] at (7,-2) {$\pi_b(x)$};

\draw [decorate,
    decoration = {brace}] (7.1,.45) --  (7.1,-2);
    \node [right] at (7.1,-.78) { $n\cdot \kappa(x)$};
\draw [dotted, thick] (0,-1.5) .. controls (3,0) and (6,.5) .. (10,.7);
% \draw [dotted] (0,-2.5) .. controls (3,-4) and (6,-4.5) .. (10,-4.7);
\draw[dotted, thick] (0,-1.5) .. controls (-.3,-1.75) .. (-.3, -2);

    \end{tikzpicture}
    \caption{A $(\K, n)$-neighborhood of geodesic $b$.}
    \label{fig:my_label}
\end{figure}

\begin{definition}[$\kappa$-fellow travelling]
    Let $\alpha$ and $\beta$ be two infinite quasi-geodesic rays in $X$. If $\alpha$ is contained in some $\kappa$-neighbourhood of $\beta$ and $\beta$ is contained in some $\kappa$-neighbourhood of $\alpha$, we say that $\alpha$ and $\beta$ \textit{$\kappa$-fellow travel} each other. This defines an equivalence relation on the set of quasi-geodesic rays in $X$. It is known that, in CAT(0) spaces, each equivalence class of quasi-geodesics contains a unique geodesic ray emanating from $\ob$. (See \cite{QRT19}). 

\end{definition}

\begin{definition}[$\kappa$-Morse geodesics]
     A geodesic $b$ is \textit{$\kappa$-Morse} if there is a function $m_{b}: \mathbb{R}_{+}^{2} \rightarrow \mathbb{R}_{+}$ so that if $\alpha:[s, t] \rightarrow X$ is a (K, C)-quasi-geodesic with end points on $b$ then

$$
\alpha[s, t] \subset \mathcal{N}_{\kappa}\left(b, m_b(K, C)\right).
$$

We refer to $m_{b}$ as the\textit{ Morse gauge} for $b$. We also always assume that $m_{b}(K, C)$ is the largest element in the set $\left\{K, C, m_{b}(K, C)\right\}$. Note that when $\kappa \equiv 1$ we recover the standard definition of a Morse geodesic. For those with independent interest of Morse geodesics, see \cite{Cor19}.
\end{definition}

The following theorem is useful for connecting the above definitions as most of the language in this paper uses $\K$-contracting instead of $\K$-Morse. In order to use the following characterization, we now assume that $X$ is a \textit{proper} CAT(0) space for the remainder of the paper.

\begin{theorem}[Theorem 3.8 in \cite{QRT19}]
Let $X$ be a proper CAT(0) space. A geodesic ray $b$ is $\K$-contracting if and only if it is $\kappa$-Morse.
\end{theorem}

\begin{definition}[Sublinearly Morse boundary]
     Let $\kappa$ be a sublinear function and let $X$ be a proper CAT(0) space. We define the \textit{$\kappa$-Morse boundary}, as a set, by

$$
\partial_{\kappa} X:=\{\text { all } \kappa \text {-Morse quasi-geodesics }\} / \kappa \text {-fellow travelling .}
$$

We postpone discussion of the topologies defined on $\p_\K X$ until Section \ref{Section 4} where such definitions will be relevant.

\begin{notation}
For the remainder of this paper, we will denote elements of $\partial_{\kappa} X$ as $a^\infty$, $b^\infty$, and so on, where $a$ and $b$ will be the unique geodesics based at $\ob$ that is in the equivalence class of $a^\infty$ and $b^\infty$, respectively. When relevant, quasi-geodesics in the same equivalence class as $a^\infty$ and $b^\infty$ will be denoted as $\alpha$ and $\beta$, respectively. We also often identify geodesics and quasi-geodesics with their images in $X$ and treat them as subset of $X$ when convenient. Lastly, we often use $C$ and $D$ to represent contracting and/or Morse constants\hypertarget{Section2.3}{.}
\end{notation}

\noindent 2.3. \textbf{ Curtain machinery.}\vspace{.25cm}

We now import many cubical results created in \cite{PSZ22}. The most important of which is the definition of a \textit{curtain} as defined below. Curtains in CAT(0) spaces become the analogue of hyperplanes in CAT(0) cube complexes we desire. The main comparison we are interested in is how geodesics crossing curtains in CAT(0) spaces mimic behavior of geodesics crossing hyperplanes in CAT(0) cube complexes. In order to show this comparison in Section \ref{Section 3}, we review the lemmas and definitions below.
 
 \begin{definition}[Curtain, Pole]\label{curtain} Let $X$ be a CAT(0) space and let $b:I\rightarrow X $ be a geodesic. For any number $r$ such that $[r-\frac{1}{2}, r+\frac{1}{2}]$ in in the interior of $I$, the \textit{curtain dual to $b$} at $r$ is 
 \begin{center}
 $h=h_b=h_{b,r}=\pi^{-1}_b(b[r-\frac{1}{2}, r+\frac{1}{2}])$
 \end{center}
where $\pi_b$ is the closest point projection to $b$. We call the segment $b[r-\frac{1}{2}, r+\frac{1}{2}]$ the \textit{pole} of the curtain which we denote as $P$ when needed.

 \end{definition}

 It is worth noting that curtains $h_{b,r}$ are defined from some geodesic $b$ at time $r$, but we often use the simpler notation $h$ when such information is not needed or already implied. There are certain properties that curtains and hyperplanes have in common. An example is that both curtains and hyperplanes separate their complements into two component \textit{half spaces} which we denote as $h^-$ and $h^+$. Also, both curtains and hyperplanes are closed as sets. 

\begin{remark}
    A notable difference between curtains and hyperplanes is that curtains are not convex, see \cite[Remark 2.4]{PSZ22} for more details.
\end{remark}

\begin{definition}[Chain, Separates]
    A curtain $h$ \textit{separates} sets $A,B \subset X$ if $A \subset h^-$ and $B \subset h^+$. A set $\{h_i\}$ is a \textit{chain} if each of the $h_i$ are disjoint and $h_i$ separates $h_{i-1}$ and $h_{i+1}$ for all $i$. We say a chain $\{h_i\}$ \textit{separates} sets $A,B \subset X$ if each $h_i$ separates $A$ and $B$.
\end{definition}

 The notion of chains separating two sets $A$ and $B$ can give a description of the distance between sets $A$ and $B$. More specifically, a maximal chain that separates two points $x,y \in X$ tells us the distance between $x$ and $y$, as shown in the following lemma.

\begin{lemma}[Lemma 2.10  in \cite{PSZ22}]
    For any $x,y \in X$, there is a chain $c$ of curtains dual to $[x,y]$ such that $1 + |c| = \lceil d(x,y) \rceil$.
\end{lemma}

Many of our arguments involve a geodesic crossing a curtain. A geodesic $b$ will \textit{cross} a curtain $h$ if there exists $s<t$ such that $b(s) $ and $b(t)$ are separated by $h$. The next lemma describes how $b$ interacts with the pole of $h$ when $b$ crosses $h$, and we use this lemma implicitly throughout the paper.

\begin{lemma}[Lemma 2.5 in \cite{PSZ22}] \label{crossing}
Let $h=h_{b, r}$ be a curtain, and let $x \in h^{-}, y \in h^{+}$. For any continuous path $\gamma:[c, d] \rightarrow X$ from $x$ to $y$ and any $t \in[r-1 / 2, r+1 / 2]$, there is some $p \in[c, d]$ such that $\pi_b (\gamma(p))=b(t).$
\end{lemma}

In particular, if a geodesic $a$ crosses a curtain $h_{b,r}$, there exists some $p \in a$ such that $\pi_b (a(p)) = b(r)$.

\begin{definition}[$L$-separated, $L$-chain] \label{L chain}  Let $L \in \mathbb{N}$. Disjoint curtains $h$ and $h^{\prime}$ are said to be \textit{$L$-separated} if every chain meeting both $h$ and $h^{\prime}$ has cardinality at most $L$. Two disjoint curtains are said to be \textit{separated} if they are $L$-separated for some $L$. If $c$ is a chain of curtains such that each pair is $L$-separated, then we refer to $c$ as an \textit{$L$-chain}. See Figure \ref{L sep} for an example of $L$-separation.
\end{definition}

\begin{definition}[$L$-metric] \label{Lmetric} Denote $X_L$ for the metric space $(X, d_L)$, where $d_L$ is the metric defined as \begin{center}
     $\mathrm{d}_L(x, y)=1+\max \{|c|: c$ is an $L$-chain separating $x$ from $y\}$
\end{center}
with $d_L(x,x) = 0$. Note that, by Remark 2.16 in \cite{PSZ22}, we have that for any $x,y \in X$, it follows that $d_L(x,y) < 1 +d(x,y).$
\end{definition}

These $X_L$ spaces will be used as auxiliary spaces to define the curtain model (See Definition \ref{Curtain Model}).

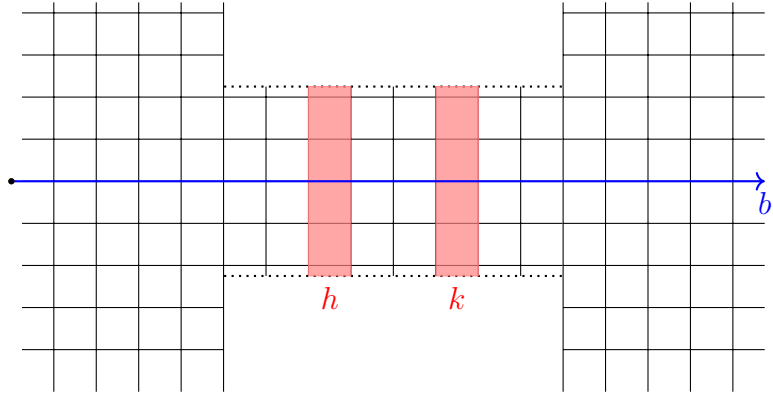
\begin{figure}[ht]
 \centering
\begin{tikzpicture}[scale = .7]

\draw[step=.8cm,black] (4.8+1.8,-3.2) grid (8+2.4,4.2);
\draw[step=.8cm,black] (8.01+2.4,-1) grid (14.39+2.4,2.6);
\draw[dotted,thick] (8.01+2.4, 2.6) to (14.39+2.4, 2.6);
\draw[dotted, thick] (8.01+2.4, -1) to (14.39+2.4, -1);
\draw[step=.8cm,black] (14.4+2.4,-3.2) grid (17.6+3,4.2);

\draw[red!50, fill = red!50, opacity=0.7  ] plot((9.6+2.4,-1) -- (10.4+2.4,-1) -- (10.4+2.4,2.6) -- (9.6+2.4, 2.6) -- cycle;
\draw[red!50, fill = red!50, opacity=0.7  ] plot((9.6+4.8,-1) -- (10.4+4.8,-1) -- (10.4+4.8,2.6) -- (9.6+4.8, 2.6) -- cycle;

\node [below, red] at (12.4,-1) {$h$};
\node [below, red] at (14.8,-1) {$k$};

\draw[step=.8cm, thick, blue, ->] (6.4,.8) to (20.6,.8);
\draw[fill] (6.4,.8) circle [radius=0.05];
\node [below, blue] at (20.6,.8) {$b$};
\end{tikzpicture}

    \caption{The above figure shows $\mathbb{R}^2$ with 2 strips cut out. We see the pair of curtains $h,k$ dual to the geodesic $b$ are 4-separated. A maximal chain that crosses both curtains would be 4 horizontal curtains. Thus, $\{h,k\}$ is a 4-chain. }
    \label{L sep}
\end{figure}

\begin{theorem}[Theorem 3.5 in \cite{PSZ22}]
\label{X_Lhyp}
For each $L < \infty$, the space $X_L$ is a quasigeodesic hyperbolic space. Moreover, $\text{Isom}\hspace{.1cm} X \leqslant \text{Isom} \hspace{.1cm} X_L$.
\end{theorem}

For a CAT(0) cube complex $X$, these $X_L$ spaces were first formed using hyperplanes instead of curtains and were also hyperbolic \cite{Gen20b}. With a sufficiently large $L$, the authors of \cite{IMZ21} use this hyperbolic space to continuously inject the sublinearly Morse boundary of $X$ into the Gromov boundary of $X_L$. However, they used the additional assumption that $X$ admits a \textit{factor system} (see \cite{BHS17}). This assumption makes these $X_L$ spaces equal for all $L\geq L_0$, where $L_0$ is some constant dependent on the hierarchical structure of $X$. Without these $X_L$ spaces stabilizing, it is possible for for $\K$-contracting geodesics to be of bounded diameter in each $X_L$ for all $L$. This the main motivation for projecting to the \textit{curtain model} (see Definition \ref{Curtain Model}) in Theorem \hyperlink{TheoremA}{A} instead of $X_L$ for some $L$.

\section{Curtain Characterization for Sublinear Contracting Rays} \label{Section 3}

This section focuses on creating a curtain characterization for $\K$-contracting rays. Such a characterization allows combinatorial arguments for the proofs building up to Theorem \hyperlink{TheoremA}{A} in Section \ref{Section 4}. Since we are working in a sublinear case, we define an sublinear analogue of an $L$-chain\hypertarget{Section3.1}{.}\\

\noindent 3.1 \textbf{A characterization via dual curtains.} \vspace{.25cm}

\begin{definition}[$\K$-chain, $\K$-curtain-excursion geodesic, excursion constant] A \textit{$\K$-chain} meeting some geodesic $b$ is a chain of curtains $\{h_i\}$ meeting $b$ at points $b(t_i)$ such that \begin{itemize}
    \item $t_{i+1}-t_i\leq C\K(t_{i+1})$
    \item $h_i$ and $h_{i+1}$ are $C\K(t_{i+1}$)-separated 
 \end{itemize}
for some $C>0$. If such $\{h_i\}$ are dual to $b$, then (up to a small increase in $C$) we choose $b(t_i)$ to be the centers of the poles of each $h_i$. A geodesic $b$ is a \textit{$\K$-curtain-excursion geodesic} when it is dual to a $\K$-chain. We refer to  $C$ as the \textit{excursion constant}.
\end{definition}

In the CAT(0) cube complex setting, a geodesic crossing a $\K$-chain of hyperplanes was defined as a $\K$-excursion geodesic in \cite{MQZ20}, so our above name of $\K$-curtain-excursion geodesics as a curtain analogue is fitting. This section works to prove Theorem \ref{curtainexcursion}, a dualized version of Theorem \hyperlink{TheoremC}{C} in the introduction. Subsection \hyperlink{Section3.2}{3.2} will recover all of Theorem \hyperlink{TheoremC}{C} by working with non-dual chains.

\begin{theorem}\label{curtainexcursion}
A geodesic ray $b$ is $\K$-contracting if and only if it is $\K$-curtain-excursion.
\end{theorem}

The forward direction is Proposition \ref{ContractingToExcursion} whereas the backward direction is Proposition \ref{ExcursionToContracting}. Before we prove the forward direction, we recall the following two lemmas.

\begin{lemma}[Lemma 2.6 in \cite{PSZ22}]
\label{starconvexity}
Let $h$ be a curtain with pole $P$. For every $x \in h$, the geodesic $\left[x, \pi_P (x)\right]$ is contained in $h$. In particular, $h$ is path-connected.
\end{lemma}

\begin{lemma}[Lemma 4.14 in \cite{MQZ20}]
\label{4.14}
Assume $X$ is a CAT(0) space. Let $b$ be a $\K$-contracting geodesic ray with contracting constant $D$ starting at $\ob$, and let $x,y\in X$ and not in $b$ such that $d(\ob, \pi_b(x)) \leq d(\ob, \pi_b(y))$. If the projection of $[x,y]$ to $b$ is larger than $4D\K(\pi_b(y))$, then $\pi_b([x,y]) \subseteq N_{5D\K(\pi_b(y))}([x,y]).$

\end{lemma}

Lemma \ref{starconvexity} tells us that for any curtain $h$ and any $x\in h$, there will exist a geodesic in $h$ connecting $x$ to the pole of $h$, and Lemma \ref{4.14} gives us that geodesics with large projections to a $\K$-contracting ray must get sublinearly close to the $\K$-contracting ray.

\begin{proposition} \label{ContractingToExcursion}
  Let $b$ be a $\K$-contracting ray with contracting constant $D$, then there exists $t_i \in \mathbb R$ such that $b$ is dual to a $\K$-chain $\{h_i\}$ at points $b(t_i) \in h_i$ and
\begin{itemize}
    \item $t_{i+1}-t_i\leq C\K(t_{i+1})$
     \item $h_i$ and $h_{i+1}$ are $C\K(t_{i+1}$)-separated 
 \end{itemize}for some $C\geq 0$ depending only on $D.$ In other words, $b$ is a $\K$-curtain-excursion geodesic.
\end{proposition}

 \begin{proof} Since $\K$ is a sublinear function, we can choose some $t_0$ such that $\K(t) \leq t$ for all $t\geq t_0$. For $i \in \mathbb{Z}_{\geq0}$, choose $t_{i+1}$ such that \begin{center}
     $t_{i+1} - t_i = 10D\K(t_{i+1})$.
 \end{center}
 
 Note this is possible to do by the nature of $\K$ being sublinear. Consider the chain $\{h_i = h_{b,i}\}$.  What is left is to show $\{h_i \}$ have the second condition of a $\K$-chain. Let $k$ be a curtain meeting both $h_i$ and $h_{i+1}$ and let $P$ be its pole. Notice $\pi_b(P)$ has diameter less than 1 since $\pi_b$ is $1-$Lipschitz.  Let $x \in h_i\cap k$, $y \in h_{i+1} \cap k$. There exists  $x',y' \in P$ such that $[x,x'],[y,y'] \subset k$ by Lemma \ref{starconvexity}. Namely, $x' = \pi_P(x)$ and $y'=\pi_P(y).$ Now, the projection of the concatenation $[x,x']*[x',y']*[y',y]$ onto $b$ will have diameter greater than $10D\K(t_{i+1}) - 1$. Since $\pi_b([x',y'])\subset \pi_b(P)$ has diameter less than 1, it must be that at least one of the projections of $[x,x']$ or $ [y,y']$ has diameter greater than $4D\K(t_{i+1})$. See Figure \ref{fig Pole}.
 
 Without loss of generality, we assume $d(\pi_b(x), \pi_b(x')) \geq 4D\K(t_{i+1})$. Thus,  there exists a $p\in[t_i +5D\K(t_{i+1}) - 1, t_i+ 5D\K(t_{i+1})+1]$ such that $p\notin \pi_b(P)$, but also $p \in \pi_b([x,x'])$. By Lemma \ref{4.14}, we have that $p$ is within $5D\K(t_{i+1}) $ of $[x,x'] \subset k$. Hence, $b(t_i+5D\K(t_{i+1}))$ is within $5D\K(t_{i+1})+1$ of $k.$ Since this is true for any curtain  $k$ meeting both $h_i$ and $h_{i+1}$, any chain that meets $h_i$ and $h_{i+1}$ must be bounded by $10D\K(t_{i+1}) +3$. This gives the well-separation bound for $\K$-excursion. 
 \end{proof}
\begin{figure}
    \centering
    \begin{tikzpicture}[scale=1.1]

    % \draw[red!50, fill = red!50, opacity=0.9  ] plot((2.2,-3) -- (3,-3) -- (3,1.5) -- (2.2,1.5) -- cycle;

    % \draw[red!50, fill = red!50, opacity=0.9  ] plot((8.2,-3) -- (9,-3) -- (9,1.5) -- (8.2,1.5) -- cycle;

    \draw[blue!50, fill = blue!50, opacity=0.4  ] plot((4.7+.4+.3,0) -- (5.7+.4+.3,.8) -- (10+.3,1) -- (9.8+.3,-.2) -- cycle;
    \draw[blue!50, fill = blue!50, opacity=0.4 ] plot((4.7+.4+.3,0) -- (5.7+.4+.3,.8) -- (1.5+.3,1) -- (1.3+.3,-.2) -- cycle;

        \fill[red, opacity=0.2] (.8+1.3,-3) .. controls(1+1.3,-2.5)  and (1+1.3,.5)   .. (.8+1.3,1.5) -- (1.8+1.3,1.5).. controls (1.6+1.3, .5)  and (1.6+1.3, -2.5) .. (1.8+1.3,-3) -- cycle;

        \fill[red, opacity=0.2] (.8+1.3+6,-3) .. controls(1+1.3+6,-2.5)  and (1+1.3+6,.5)   .. (.8+1.3+6,1.5) -- (1.8+1.3+6,1.5).. controls (1.6+1.3+6, .5)  and (1.6+1.3+6, -2.5) .. (1.8+1.3+6,-3) -- cycle;

% geodesic b
    \draw[blue] (4.7+.4+.3,0) to (5.7+.4+.3,.8);
    \draw [,thick, ->](1,-2) to (10,-2);
    \node [below] at (10,-2) {$b$};

    \draw[fill, opacity=0.5 ] (2.7,.7) circle [radius=0.05];
    \node [left] at (2.7,.7) {\small $x$};

    \draw[fill, opacity=0.5] (8.4+.3,.1) circle [radius=0.05];
    \node [right, opacity=0.5] at (8.4+.3,.1) {\small $y$};

    \draw[fill, opacity=0.5] (5.35+.3,.2) circle [radius=0.05];
    \node [right, opacity=0.5] at (5.35+.3,.2) {\small $x'$};

    \draw[fill, opacity=0.5] (5.8+.35,.6) circle [radius=0.05];
    \node [left, opacity=0.5] at (5.85+.3,.6) {\small $y'$};

    \draw[ opacity=0.5] (2.7,.7) to (5.35+.3, .2);
    \draw[ opacity=0.5] (8.4+.3,.1) to (5.85+.3,.6);

    \node  at (4.7,-2) { $[$};
    \node  at (6.5,-2) { $]$};

    \node[blue, above] at (6.4, .8) {$P$};
    \node[blue, left] at (1.6,.4) {$k$};
    \node[red, above] at (2.6,1.5) {$h_i$};
    \node[red, above] at (8.6,1.5) {$h_{i+1}$};

    \node[below] at (2.6, -2) {$b(t_i)$};
    \draw[fill] (2.6,-2) circle [radius=0.05];

    \node[below] at (8.6, -2) {$b(t_{i+1})$};
    \draw[fill] (8.6,-2) circle [radius=0.05];

    \draw[blue!50, fill = blue!50, opacity=0.5  ] plot((5.4,-2.05) -- (5.4,-1.95) -- (6.4,-1.95) -- (6.4,-2.05) -- cycle;
    \node[below, blue] at (5.9,-2) { \small $\pi_a(P)$};

    \end{tikzpicture}
    \caption{A picture of the argument of Proposition \ref{ContractingToExcursion}. We have $h_i$ and $h_{i+1}$ are curtains dual to $b$ such that $d\left(b\left(t_i\right),b\left(t_{i+1}\right)\right) = 10D\K(t_{i+1})$, $k$ is a curtain crossing both $h_i$ and $h_{i+1}$, and $P$ is the pole of $k$. Note that, no matter where $P$ projects down to $b$, $\pi_b\left(\left[x,x'\right]*\left[x',y'\right]*\left[y',y\right]\right)$ will contain $\left[b\left(t_i+\frac{1}{2}\right), b\left(t_{i+1}-\frac{1}{2}\right)\right]$ and will be of length at least $10\K(t_{i+1}) - 1.$}
    \label{fig Pole}
\end{figure}
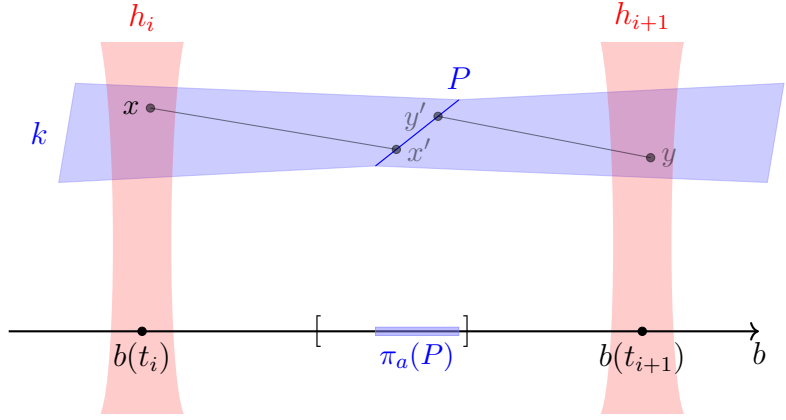

 \begin{remark}
 We have to double our bound since it is possible to have two chains of length $5D\K(t_{i+1})+1$ that are disjoint from each other. For example, in Figure \ref{fig Pole} there can be one chain ``above" $b$ and another chain ``below" $b$.
\end{remark}
 The above argument follows from the argument of \cite[Theorem 4.2]{PSZ22}, which showed that a $D$-contracting ray has a $(10D+3)$-chain of curtains dual to the ray. The reverse direction follows arguments in \cite{MQZ20}. The following lemma shows that certain curtains do not create bigons, and it will be used in Lemma \ref{closeness}.

 \begin{lemma}[Lemma 2.7 in \cite{PSZ22}]
\label{nobigons}
Let $b=\left[x_1, x_3\right]$ be a geodesic and let $x \notin b$. For any $x_2 \in b$, if $h$ is a curtain dual to $\left[x_2, x\right]$ that meets $\left[x_1, x_2\right]$, then $h$ does not meet $\left[x_2, x_3\right]$.
\end{lemma}

 \begin{lemma}
 \label{closeness}
 Let $b$ be a $\K$-curtain-excursion geodesic ray with excursion constant $C>0$ and  $\K$-chain $\{h_i\}$.  If $a$ is another geodesic and crosses $h_{i-1}, h_i, h_{i+1}$ with  $\pi_b(a(s_i)) = b(t_i)$ for some $s_i >0$, then there exists a $D>0$ depending only on $C$ such that $d(a(s_i), b(t_{i})) \leq D\K(t_{i})$. 
 \end{lemma}

 \begin{proof}
 The following argument is illustrated in Figure  \ref{fig3}. Let $s_{i-1} < s_i <s_{i+1}$ such that $\pi_b(a(s_j)) =  b(t_j)$ for $j \in \{i-1, i, i+1\}.$ Consider geodesics $[a(s_i),a(s_{i+1})],$ $ [a(s_{i+1}),b(t_{i+1})],$ $ [b(t_{i}), b(t_{i+1})],$ and $[a(s_i), b(t_{i})].$ This is a quadrilateral in our space. Let $c$ be a maximal chain dual to $[a(s_i), b(t_{i})]$. Note that all curtains in $c$ must cross at least one of $[a(s_i),a(s_{i+1})],$ $ [a(s_{i+1}),b(t_{i+1})],$ or $ [b(t_{i}), b(t_{i+1})].$ Thus, denote $c_1$, $c_2$, and $c_3$ as the collections of curtains in $c$ as seen in Figure \ref{fig3}. That is, $c_1$ are the curtains in $c$ that also meet $ [b(t_{i}), b(t_{i+1})]$ and so on. We get $|c| \leq |c_1|+|c_2|+|c_3|$. Also, define $c'$ as the curtains in $c_3$ that also meet $[a(s_{i-1}),b(t_{i-1})]$  and $c''$ as the curtains of $c_3$ that also meet $[b(t_{i-1}),b(t_{i+1})]$. No curtains in $c_3$ meet $[a(s_{i-1}),a(s_i)]$ by Lemma \ref{nobigons}. We have that $|c_3| \leq |c'| + |c''|$.

 Now, $|c_1| \leq C\K(t_{i+1})$ since, by assumption of the $\K$-curtain-excursion chain, the length of  $[b(t_{i}), b(t_{i+1})]$ is bounded by $C\K(t_{i+1})$. Also, we have $|c_2| \leq C\K(t_{i+1})$ because  $h_i$ and $h_{i+1}$ are  $C\K(t_{i+1})$-separated. Similarly, we have $|c'|\leq C\K(t_{i})$, and $|c''|\leq C\K(t_{i})$  This gives 
 
 \begin{align*}
     |c| &\leq |c_1|+|c_2|+|c_3| \\
         &\leq |c_1|+|c_2|+|c'| +|c''|\\
         &\leq C\K(t_{i+1})+ C\K(t_{i+1})+ C\K(t_{i})+ C\K(t_{i})\\
         &\leq 4C\K(t_{i+1}).\\
 \end{align*}

 \begin{figure}[ht]
    \centering
    \begin{tikzpicture}[scale=1]

     % \draw[red]  (.8,-3) .. controls(1,-2.5)  and (1,.5)   .. (.8,2.5);
     % \draw[red] (1.8,-3) .. controls (1.6, -2.5) and (1.6, .5) .. (1.8,2.5);
    \fill[red, opacity=0.2] (.8,-3) .. controls(1,-2.5)  and (1,.5)   .. (.8,2.5) -- (1.8,2.5).. controls (1.6, .5)  and (1.6, -2.5) .. (1.8,-3) -- cycle;

    \fill[red, opacity=0.2] (.8+3.1,-3) .. controls(1+3.1,-2.5)  and (1+3.1,.5)   .. (.8+3.1,2.5) -- (1.8+3.1,2.5).. controls (1.6+3.1, .5)  and (1.6+3.1, -2.5) .. (1.8+3.1,-3) -- cycle;

    \fill[red, opacity=0.2] (.8+6.2,-3) .. controls(1+6.2,-2.5)  and (1+6.2,.5)   .. (.8+6.2,2.5) -- (1.8+6.2,2.5).. controls (1.6+6.2, .5)  and (1.6+6.2, -2.5) .. (1.8+6.2,-3) -- cycle;

    % \draw[red!50, fill = red!50, opacity=0.9  ] plot((2-1,-3) -- (2.6-1,-3) -- (2.6-1,2.5) -- (2-1,2.5) -- cycle;
        
    % \draw[red!50, fill = red!50, opacity=0.9 ] plot((4.6-.5,-3) -- (5.2-.5,-3) -- (5.2-.5,2.5) -- (4.6-.5,2.5) -- cycle;
    % \draw[red!50, fill = red!50, opacity=0.9 ] plot((7.2,-3) -- (7.8,-3) -- (7.8,2.5) -- (7.2,2.5) -- cycle;

%point values of b
    \draw[fill] (1.3,-2) circle [radius=0.03];
    \node [below] at (1.3,-2) {\footnotesize$b(t_{i-1})$};

    \draw[fill] (4.9-.5,-2) circle [radius=0.03];
    \node [below, ] at (4.9-.5,-2) {\footnotesize $b (t_{i})$};
    
    \draw[fill] (7.5,-2) circle [radius=0.03];
    \node [below, ] at (7.5,-2) {\footnotesize $b (t_{i+1})$};
    
% geodesic b and a
    \draw [->] (-1,1) .. controls(3,-1) and (7,-1) ..  (9,1);
    \draw [blue, ->](-1,-2) to (10,-2);
    \node [below, blue] at (10,-2) {$b$};
    \node[above] at (9,1) {$a$};

%vertices
    \draw[fill] (1.3,.08) circle [radius=0.03];
    \draw[fill] (4.4,-.49) circle [radius=0.03];
    \draw[fill] (7.5,.02) circle [radius=0.03];

%nodes
    \node [above] at (1.3, .08) {$a(s_{i-1})$};
    \node [above] at (4.4, -.49) {$a(s_i)$};
    \node [above] at (7.5, .02) {$a(s_{i+1})$};
    \node [above, red, opacity=0.9 ] at (1.3,2.5) {\large $h_{i-1}$};
    \node [above, red, opacity=0.9 ] at (4.4,2.5) {\large $h_{i}$};
    \node [above, red, opacity=0.9 ] at (7.5,2.5) {\large $h_{i+1}$};
    \node[left] at (3.6, -1.35+.4) {$c_3$};
    \node[left] at (3.6, -1.6+.22) {$c_2$};
    \node[left] at (3.6, -1.85+.1) {$c_1$};

%projection lines
    \draw[dashed] (4.4, -.51) to (4.4,-2);
    \draw[dashed] (7.5, .02) to (7.5,-2);
    \draw[dashed] (1.3, 0.1) to (1.3,-2);
    
%c1, c2, c3 curtains
    \draw[black!25!green, thick] (3.6, -1.35 +.4) .. controls(5.5,-1.35+.4) .. (6, 1+.4);
    \draw[black!25!green, thick] (3.45, -1.3+.4) .. controls (3.7,-1.3+.4) and (3.7,-1.4+.4) .. (3.45, -1.4+.4);
    \draw [black!25!green, thick] (3.5, -1.6+.22) to (8.2, -1.6+.22);
    \draw [black!25!green, thick] (3.5, -1.85+.1) .. controls(5.5,-1.85+.1) .. (6, -3+.1);

    \draw[black!25!green, thick] (3.1, -1.3+.4) to (.5, -1.3+.4);
    \draw[black!25!green, thick] (3.1, -1.4+.4) .. controls(2.6,-1.4+.4) .. (2.4, -3.5+.4);
 %c' c'' curtains   
    \node[left] at (.5, -.9) {$c'$};
    \node[below] at (2.4, -3.1) {$c''$};

    \end{tikzpicture}
    \caption{The set up of Lemma \ref{closeness} with subchains $c_1,c_2,$ and $c_3$ included. The bounds of $c_1,c_2$ and $c_3$ will show that $d(a(s_i), b(t_{i})) \leq D\K(t_{i})$ for some $D$ depending on $C$.}
    \label{fig3}
\end{figure}
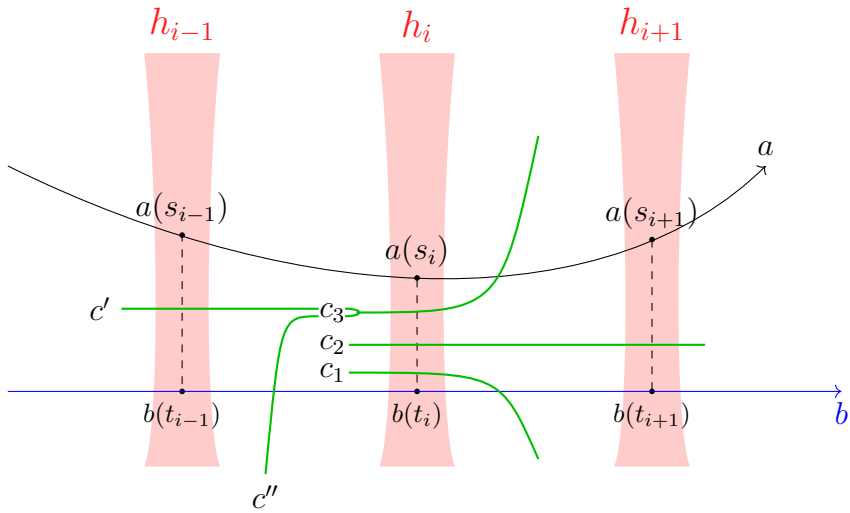
 Since $|c|$ is a maximal chain of $[a(s_i), b(t_{i})]$, we have that $a(s_i)$ is within a distance of $(4C\K(t_{i+1}) + 1)$ from $b(t_{i})$. By Lemma \ref{3.2}, we have that there exists a $D\geq 0$ such that $d(a(s_i), b(t_{i})) \leq D\K(t_{i}).$ 
 \end{proof}

The following definition from \cite{MQZ20} gives another characterization of $\K$-contracting rays that will be useful in the reverse direction of Theorem \ref{curtainexcursion}

\begin{definition}[$\K$-slim geodesic]
We say an infinite geodesic ray $b$ is \textit{$\K$-slim} if there exists some $C\geq 0$ such that for any $x \in X$, $y \in b,$ we have $d(\pi_b(x), [x,y]) \leq C\K(\pi_b(x))$.
\end{definition}

\begin{lemma}[Proposition 3.6/Corollary 3.7 in \cite{MQZ20}] \label{kslim} A geodesic ray $b$ is $\K$-contracting if and only if it is $\K$-slim.

\end{lemma}

\begin{proposition} \label{ExcursionToContracting}
  Let $b$ be a $\K$-curtain-excursion geodesic with $\K$-chain $\{h_i\}$ dual to $b$. Denote $b(t_i)$ as the centers of the poles of each $h_i$, and put $C\geq0$ the excursion constant. Then $b$ is $\K$-slim for constant $D'\geq0$, depending only on $C$ (and $\K$-contracting by Lemma \ref{kslim}).
\end{proposition}

\begin{proof}
 Let $x \in X$. Then, there exists a minimal $i$ such that $x\in h_i^-$. Let $y \in b$. We split where $y$ can be placed on $b$ into three cases: $y \in h_{i+2}^+$, $y \in h^-_{i-4}$ and $y \in [b(t_{i-4} - \frac{1}{2}), b(t_{i+2}+\frac{1}{2})]$. 

 When $y \in h_{i+2}^+$, take $[x,y]$ which will cross $h_i,h_{i+1}$ and $h_{i+2}$. By Lemma \ref{closeness}, we have $b(t_{i+1})$ is within $D\K(t_{i+1})$ of $[x,y]$ for some $D>0$ depending only on $C$. We know $\pi_b(x) \in h_{i-2}^+$. Thus, $d(\pi_b(x), b(t_{i+1})) \leq 3C\K(t_{i+1})$. This gives $$d(\pi_b(x), [x,y]) \leq (3C+D)\K(t_{i+1}).$$  Since $\pi_b(x)$ is on the geodesic $[b(t_{i-2}),b(t_i)]$, we get that Lemma \ref{3.2} implies there exists a $D'$ such that $d(\pi_b(x), [x,y]) \leq D'\K(\pi_b(x))$. See Figure \ref{K slim fig}.  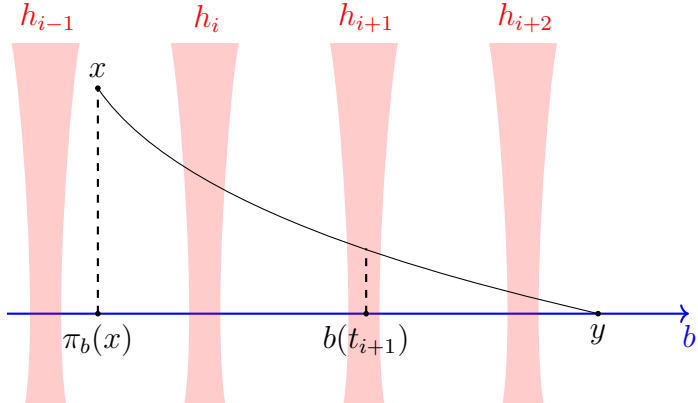
\begin{figure}[ht]
     \centering
\begin{tikzpicture}[scale=.6]

% %help lines 
%     \draw[step=1cm,gray, opacity=.5] (0,-2) grid (15,10);
    \draw[blue, thick, ->] (0,0) to (15,0);

    \node [below, blue] at (15,0) {$b$};

    \draw[fill] (2,5) circle [radius= 0.05];
    \draw[fill] (13,0) circle [radius= 0.05];

    % \draw[red]  (3.9-3.5,-2) .. controls(4.2-3.5,0)  and (3.8-3.5,5-1)   .. (3.6-3.5,6);
    % \draw [red] (5.2-.4-3.5, -2) ..controls (5-.4-3.5,0) and (5.2-.4-3.5,5-1)  .. (5.4-.4-3.5,6); 

    \fill[red, opacity=0.2] (3.9-3.5,-2) .. controls(4.2-3.5,0)  and (3.8-3.5,5)   .. (3.6-3.5,6) -- (5.5-.4-3.5,6) .. controls (5.3-.4-3.5,5)  and (4.9-.4-3.5,0) .. (5.2-.4-3.5, -2) -- cycle;

    % \draw[red]  (3.9,-2) .. controls(4.2,0)  and (3.8,5)   .. (3.6,6);
     %\draw [red] (5.2-.4, -2) ..controls (4.9-.4,0) and (5.3-.4,5)  .. (5.5-.4,6); 

    \fill[red, opacity=0.2] (3.9,-2) .. controls(4.2,0)  and (3.8,5)   .. (3.6,6) -- (5.5-.4,6) .. controls (5.3-.4,5)  and (4.9-.4,0) .. (5.2-.4, -2) -- cycle;

   % \draw[red]  (3.9+3.5,-2) .. controls(4.2+3.5,0)  and (3.8+3.5,5)   .. (3.6+3.5,6);
   %  \draw [red] (5.2-.4+3.5, -2) ..controls (5-.4+3.5,0) and (5.2-.4+3.5,5)  .. (5.4-.4+3.5,6); 

    \fill[red, opacity=0.2] (3.9+3.5,-2) .. controls(4.2+3.5,0)  and (3.8+3.5,5)   .. (3.6+3.5,6) -- (5.5-.4+3.5,6) .. controls (5.3-.4+3.5,5)  and (4.9-.4+3.5,0) .. (5.2-.4+3.5, -2) -- cycle;

    % \draw[red]  (3.9+7,-2) .. controls(4.2+7,0)  and (3.8+7,5)   .. (3.6+7,6);
    % \draw [red] (5.2-.4+7, -2) ..controls (5-.4+7,0) and (5.2-.4+7,5)  .. (5.4-.4+7,6); 

    \fill[red, opacity=0.2] (3.9+7,-2) .. controls(4.2+7,0)  and (3.8+7,5)   .. (3.6+7,6) -- (5.5-.4+7,6) .. controls (5.3-.4+7,5)  and (4.9-.4+7,0) .. (5.2-.4+7, -2) -- cycle;

    \draw (2,5) .. controls(4,2) and (11,.5) ..  (13,0);

    \node [above] at (2,5) {$x$};
    \node [below] at (13,0) {$y$};
    \draw[fill] (2,0) circle [radius= 0.05];
    \node [below] at (2,0) {$\pi_b(x)$};

    \node [above, red] at (.9,6) {$h_{i-1}$};
    \node [above, red] at (4.4,6) {$h_{i}$};
    \node [above, red] at (7.9,6) {$h_{i+1}$};
    \node [above, red] at (11.4,6) {$h_{i+2}$};
    \draw[fill] (7.9,0) circle [radius= 0.05];
    \node [below] at (7.9,0) {$b(t_{i+1})$};
    \draw[dashed, thick] (7.9,0) to (7.9,1.45);

    \draw[dashed, thick] (2,0) to (2,5);

\end{tikzpicture}
     \caption{Picture of the first case of Proposition \ref{ExcursionToContracting}.  Since $d(b(t_{i+1}), [x,y]) \leq D\K(t_{i+1})$ and  $d(\pi_b(x), b(t_{i+1})) \leq 3C\K(t_{i+1})$, we see that the $\pi_b(x)$ will also be sublinearly close to $[x,y]$.}
     \label{K slim fig}
 \end{figure}

 The proof in the case of $y \in h^-_{i-4}$ is the same and its corresponding picture will be a mirrored version of Figure \ref{K slim fig}.

 Lastly, when $y \in [b(t_{i-4} - \frac{1}{2}), b(t_{i+2}+\frac{1}{2})]$, we similarly have that $\pi_b(x) \in h_{i-2}^+\cap h_i^-$ will be within $5C\K(t_{i+2})$ of $y$. Thus, again by Lemma \ref{3.2}, there will exist a $D'$ depending only on $C$ such that $d(\pi_b(x), [x,y]) \leq D'\K(\pi_b(x))$. Hence the proof (and also the proof of Theorem \ref{curtainexcursion})\hypertarget{Section3.2}{.} \end{proof}

\noindent 3.2 \textbf{Dualizing a $\K$-chain.} \vspace{.25cm} \\ 
In \cite{PSZ22}, the authors shows that if a chain of $L$-separated curtains \textbf{\textit{meets}} a geodesic, one can follow the process of \cite[Lemma 4.5]{PSZ22} to find a chain of $L$-separated curtains dual to the geodesic. If one were to follow this process in the sublinear case, it is likely that the dual curtains will be at a $\K^2$ distance apart resulting in a dual $\K^2$-chain instead of a $\K$-chain. Instead, we rework Lemma \ref{closeness} and Proposition \ref{ExcursionToContracting} to allow for a $\K$-chain that is not necessarily dual to the geodesic. This is Proposition \ref{Closeness(Not dual)} and Proposition \ref{NonDualto Contracting}, respectively. Then, Proposition \ref{ContractingToExcursion} finds a $\K$-chain dual to the geodesic.

\begin{proposition} \label{Closeness(Not dual)}
     Let $b$ be a geodesic ray that meets a $\K$-chain $\{h_i\}$ (not necessarily dual) at points $b(t_i)$ with excursion constant $C>0$. Put $P_i$ as the poles of each $h_i$. If $a$ is another geodesic and crosses $h_{i-1}, h_i, h_{i+1}$ with  $\pi_{P_i}(a(s_i)) = \pi_{P_i}(b(t_i))$ for some $s_i >0$, then there exists a $D>0$ depending only on $C$ such that $d(a(s_i), b(t_{i})) \leq D\K(t_{i})$. 
\end{proposition}

\begin{proof}
    Again, similar to Lemma \ref{closeness}, consider the geodesic $[a(s_i),b(t_i)]$ and let $c$ be a maximal chain dual to $[a(s_i),b(t_i)]$. Since $\pi_{P_i}(a(s_i)) = \pi_{P_i}(b(t_i))$, then due to Lemma \ref{starconvexity}, the concatenation $[a(s_i),\pi_{P_i}(b(t_i))]*[\pi_{P_i}(b(t_i)),b(t_i)]$ will be inside of $h_i$. Since any curtain in the chain $c$ must meet $[a(s_i),\pi_{P_i}(b(t_i))]*[\pi_{P_i}(b(t_i)),b(t_i)]$, we must have that any curtain in $c$ must meet $h_i$ (even though $[a(s_i),b(t_i)]$ might not stay inside of $h_i$). On the other hand, all curtains in $c$ must also meet the concatenation $[a(s_i),a(s_{i+1})]*[a(s_{i+1}),b(t_{i+1})]*[b(t_{i+1}),b(t_i)]$. So, like Lemma \ref{closeness}, one gets that $d(a(s_i),b(t_i) \leq D \K(t_i)$ where $D$ depends only on $C$.
\end{proof}

\begin{proposition} \label{NonDualto Contracting}
  Let $b$ be geodesic that meets a $\K$-chain $\{h_i\}$ (not necessarily dual) at points $b(t_i)$ with excursion constant $C>0$. Then $b$ is $\K$-slim for constant $D'\geq0$, depending only on $C$ (and $\K$-contracting by Lemma \ref{kslim}).
\end{proposition}

\begin{proof}

 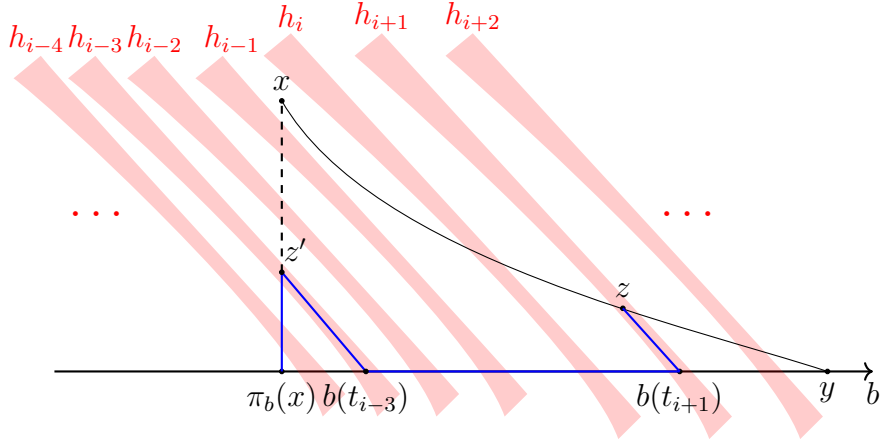
\begin{figure}[ht]
     \centering
\begin{tikzpicture}[scale=.6]

% % %help lines 
%     \draw[step=1cm,gray, opacity=.5] (-4,-2) grid (15,10);
    \draw[thick, ->] (-3,-.5) to (15,-.5);

    \node [below] at (15,-.5) {$b$};

    \draw[fill] (2,5.5) circle [radius= 0.05];
    \draw[fill] (14,-.5) circle [radius= 0.05];

    \fill[red, opacity=0.2] (3.9+3.5+2,-2) .. controls(4.2+3.5+2-2+1,0)  and (3.8+3.5+2-7+1,5)   .. (3.6+3.5+2-.5-7.5+1+.5-1,6-.5+1) -- (5.5-.4+3.5+2-.4-.5 - 8+1+.5-1,6+1) .. controls (5.3-.4+3.5+2-.4 -7+1,5)  and (4.9-.4+3.5+2-.4-2+1,0) .. (5.2-.4+3.5+2-.4, -1.5) -- cycle;

     \fill[red, opacity=0.2] (3.9+3.5+2+2,-2) .. controls(4.2+3.5+2-2+1+2,0)  and (3.8+3.5+2-7+1+2,5)   .. (3.6+3.5+2-.5-7.5+1+.5-1+2,6-.5+1) -- (5.5-.4+3.5+2-.4-.5 - 8+1+.5-1+2,6+1) .. controls (5.3-.4+3.5+2-.4 -7+1+2,5)  and (4.9-.4+3.5+2-.4-2+1+2,0) .. (5.2-.4+3.5+2-.4+2, -1.5) -- cycle;

      \fill[red, opacity=0.2] (3.9+3.5+2+4,-2) .. controls(4.2+3.5+2-2+1+4,0)  and (3.8+3.5+2-7+1+4,5)   .. (3.6+3.5+2-.5-7.5+1+.5-1+4,6-.5+1) -- (5.5-.4+3.5+2-.4-.5 - 8+1+.5-1+4,6+1) .. controls (5.3-.4+3.5+2-.4 -7+1+4,5)  and (4.9-.4+3.5+2-.4-2+1+4,0) .. (5.2-.4+3.5+2-.4+4, -1.5) -- cycle;

    \fill[red, opacity=0.2] (3.9+3.5+2-2.5,-2+.5) .. controls(4.2+3.5+2-2+1-2.5,0+.5)  and (3.8+3.5+2-7+1-2.5,5+.5)   .. (3.6+3.5+2-.5-7.5+1+.5-1-2.5+1,6-.5+1-1+.5) -- (5.5-.4+3.5+2-.4-.5 - 8+1+.5-1-2.5+1,6+1-1+.5) .. controls (5.3-.4+3.5+2-.4 -7+1-2.5,5+.5)  and (4.9-.4+3.5+2-.4-2+1-2.5,0+.5) .. (5.2-.4+3.5+2-.4-2.5, -1.5+.5) -- cycle;

    \fill[red, opacity=0.2] (3.9+3.5+2-4,-2+.5) .. controls(4.2+3.5+2-2+1-4,0+.5)  and (3.8+3.5+2-7+1-4,5+.5)   .. (3.6+3.5+2-.5-7.5+1+.5-1-4+1,6-.5+1-1+.5) -- (5.5-.4+3.5+2-.4-.5 - 8+1+.5-1-4+1,6+1-1+.5) .. controls (5.3-.4+3.5+2-.4 -7+1-4,5+.5)  and (4.9-.4+3.5+2-.4-2+1-4,0+.5) .. (5.2-.4+3.5+2-.4-4, -1.5+.5) -- cycle;

    \fill[red, opacity=0.2] (3.9+3.5+2-4-1.5+.2,-2+.5) .. controls(4.2+3.5+2-2+1-4-1.5+.2,0+.5)  and (3.8+3.5+2-7+1-4-1.5+.2,5+.5)   .. (3.6+3.5+2-.5-7.5+1+.5-1-4+1-1.5+.2,6-.5+1-1+.5) -- (5.5-.4+3.5+2-.4-.5 - 8+1+.5-1-4+1-1.5+.2,6+1-1+.5) .. controls (5.3-.4+3.5+2-.4 -7+1-4-1.5+.2,5+.5)  and (4.9-.4+3.5+2-.4-2+1-4-1.5+.2,0+.5) .. (5.2-.4+3.5+2-.4-4-1.5+.2, -1.5+.5) -- cycle;

    \fill[red, opacity=0.2] (3.9+3.5+2-4-1.5-1,-2+.5) .. controls(4.2+3.5+2-2+1-4-1.5-1,0+.5)  and (3.8+3.5+2-7+1-4-1.5-1,5+.5)   .. (3.6+3.5+2-.5-7.5+1+.5-1-4+1-1.5-1,6-.5+1-1+.5) -- (5.5-.4+3.5+2-.4-.5 - 8+1+.5-1-4+1-1.5-1,6+1-1+.5) .. controls (5.3-.4+3.5+2-.4 -7+1-4-1.5-1,5+.5)  and (4.9-.4+3.5+2-.4-2+1-4-1.5-1,0+.5) .. (5.2-.4+3.5+2-.4-4-1.5-1, -1.5+.5) -- cycle;

     % \fill[white](-4,0) -- (-4,6.5) -- (1.5,6.5) -- cycle;

    \draw (2,5.5) .. controls(4,2) and (11,.5) ..  (14,-.5);

    \node [above] at (2,5.5) {$x$};
    \node [below] at (14,-.5) {$y$};
    \draw[fill] (2,-.5) circle [radius= 0.05];
    \node [below] at (2,-.5) {$\pi_b(x)$};

    \node [above, red] at (-3.4,6.3) {$h_{i-4}$};
    \node [above, red] at (-2.1,6.3) {$h_{i-3}$};
    \node [above, red] at (-.8,6.3) {$h_{i-2}$};
    \node [above, red] at (.9,6.3) {$h_{i-1}$};
    \node [above, red] at (2.2,6.8) {$h_{i}$};
    \node [above, red] at (4.2,6.8) {$h_{i+1}$};
    \node [above, red] at (6.2,6.8) {$h_{i+2}$};
    \draw[fill] (10.75,-.5) circle [radius= 0.05];
    \node [below] at (10.75,-.5) {$b(t_{i+1})$};
    \draw[fill] (9.5,.9) circle [radius= 0.05];
    \node [above] at (9.5,.9) {$z$};
    
    \draw[ thick, blue] (10.75,-.5) to (9.5,.9);
    \draw[ thick, blue] (10.75,-.5) to (3.85,-.5);
    \draw[ thick, blue] (3.85,-.5) to (2,1.7);
    \draw[ thick, blue] (2,-.5) to (2,1.7);

    \draw[fill] (2,1.7) circle [radius= 0.05];
    \node [above] at (2.3,1.7) {$z'$};

    \draw[fill] (3.85,-.5) circle [radius= 0.05];
    \node [below] at (3.85,-.5) {$b(t_{i-3})$};

    \node [above, red] at (-2,2.5) {\Large$\cdots$};
    \node [above, red] at (11,2.5) {\Large$\cdots$};

    \draw[dashed, thick] (2,1.7) to (2,5.5);

\end{tikzpicture}
\caption{A picture for the proof of Proposition \ref{NonDualto Contracting}. No matter how many curtains in our $\K$-chain meet $[x,\pi_b(x)]$, we get that $z'$ and $b(t_{i-3})$ will still be bounded by $\K$. The blue path represents a path from $\pi_b(x)$ to $z$, and all geodesic subpaths of the blue path are bounded by $\K$ via the proof of Proposition \ref{NonDualto Contracting}. }
    \label{fig:NonDual Curtains}
\end{figure}

Again we follow arguments in Proposition \ref{ExcursionToContracting}. Let $x \in X$ and $y \in b$. There exists a minimal $i$ such that $x \in h^-_i$. To prove $b$ is $\K$-slim, we again split where $y \in b$ can be placed on $b$ into three cases: $y \in h_{i+2}^+$, $y \in h^-_{i-4}$, and $y \in [b(t_{i-4} - \frac{1}{2}), b(t_{i+2}+\frac{1}{2})]$. 

When $y \in h_{i+2}^+$, we get by Proposition \ref{Closeness(Not dual)} that there exists $z \in [x,y]\cap h_{i+1}$ such that $d(z, b(t_{i+1}))\leq D \K(t_{i+1}) $ where $D$ depends on $C$. Now, since the $\{h_i\}$ are not necessarily dual to $b$, it is possible that $[x, \pi_b(x)]$ will cross elements in our $\K$-chain $\{h_i\}$. This gives two subcases. In the subcase that $\pi_b(x) \in h^+_{i-5}$ Then $d(\pi_b(x), b(t_{i+1}) \leq 5C\K(t_{i+1})$ and so $d(\pi_b(x), z) \leq (5C+D)\K(t_{i+1}) \leq D' \K(\pi_b(x))$ for some $D'$ depending on $C$ by Lemma \ref{3.2}. In the subcase that $\pi_b(x) \notin h^+_{i-5}$, then $[x,\pi_b(x)]$ crosses $h_{i-4},h_{i-3}, h_{i-2}$ and, by Proposition \ref{Closeness(Not dual)}, there exists a $z' \in [x,\pi_b(x)]\cap h_{i-3}$ such that $d(z', b(t_{i-3}) \leq D\K(t_{i-3})$ for some $D$ depending on $C$. Since $\pi_b(z') = \pi_b(x)$, it must be true that also $d(z', \pi_b(x)) \leq D \K(t_{i-3}).$ Thus, we get, 

\begin{align*}
    d(\pi(x), z) &\leq d(\pi(x), z') + d(z', b(t_{i-3})) +d(b(t_{i-3}),b(t_{i+1})) + d(b(t_{i+1}), z)\\
    &\leq D\K(t_{i-3}) + D\K(t_{i-3}) + 4C\K(t_{i+1}) + C\K(t_{i+1})\\
    &\leq (2D+5C)\K(t_{i+1}).
\end{align*} This gives, by Lemma \ref{3.2}, there exists a $D''$ such that  $d(\pi(x), z) \leq D''\K(\pi_b(x))$. The cases of  $y \in h^-_{i-4}$ and $y \in [b(t_{i-4} - \frac{1}{2}), b(t_{i+2}+\frac{1}{2})]$ are done with a similar argument of the above case and Lemma \ref{ExcursionToContracting}, so we will omit them. See Figure \ref{fig:NonDual Curtains}.
\end{proof}

\begin{corollary}\label{dualizing excursion }
    If a $\K$-chain $\{h_i\}$ meets a geodesic $b$ at points $\{b(t_i)\}$ such that $t_{i+1}-t_i \leq C\K(t_{i+1})$ for some $C>0$, then $b$ is a $\K$-curtain-excursion geodesic.
\end{corollary}

\begin{proof}
    Proposition \ref{NonDualto Contracting} shows $b$ is $\K$-contracting, and Proposition \ref{ContractingToExcursion} finds a dual $\K$-chain to $b$.
\end{proof}

\section{Continuously Injecting into a Hyperbolic Boundary}\label{Section 4}

4.1 \textbf{ The Curtain Model and Projecting Geodesics.}\vspace{.25cm}

We now introduce a hyperbolic space that $\p_\K X$ can continuously inject into as formulated in \cite{PSZ22}. As a set, we fix the CAT(0) space $X$ but change the metric.

\begin{definition}[Curtain Model] \label{Curtain Model}
We consider the space $(X, \hat{d})$ where the distance between two points $x,y\in X$ is defined by \begin{center}
    $\hat{d}(x,y) = \displaystyle \sum_{L=1}^\infty \frac{d_L(x,y)}{L^3}$
\end{center}
where $d_L$ is the $L$-metric defined in Definition \ref{Lmetric}. We call $(X, \hat{d})$ the \textit{curtain model} of $X$ and denote it as $\widehat{X}$ as seen in the introduction.
\end{definition}

The function $\hd$ is indeed a metric since each of the $d_L$'s are metrics. Also, every isometry of $X$ is also an isometry of $X_L$ by Theorem \ref{X_Lhyp}. This implies also that any isometry of $X$ is an isometry of $\widehat{X}$. Hence $\text{Isom}\hspace{.1cm} X$ acts on $\widehat{X}$ by isometries \cite[Lemma 9.2]{PSZ22}.

\begin{remark}
    In \cite{PSZ22}, the curtain model metric used was $D(x,y) = 
    \sum_{L=1}^\infty \lambda_Ld_L(x,y)$ where $\lambda_L \in (0,1)$ such that $\sum_{L=1}^\infty \lambda_L < \sum_{L=1}^\infty L\lambda_L < \infty$. This could open the door for generalities whenever geometric properties are due to a family of metrics. However, for computations in proofs, it is nice to work with a well known series such as $\sum_{L=1}^\infty \lambda_L = \sum_{L=1}^\infty \frac{1}{L^3}$. Also, choosing $\lambda_L = \frac{1}{L^3}$ allows our theory to work when regarding all $\K$ such that $\K^4$ is sublinear. In fact, for any  $n > 2$, choosing the sequence $\lambda'_L = \frac{1}{L^n}$ results in the necessity for $\K$ to be such that $\K^{n+1}$ is sublinear. (See \cite[Remark 9.1]{PSZ22} for the requirement of $n>2$ and the proof of Lemma \ref{unbnd} for the requirement that $\K^{n+1}$ be sublinear.)
\end{remark}

\begin{theorem}[Theorem 9.10 in \cite{PSZ22}]
There exists a $\delta$ such that $\widehat{X}$ is a quasigeodesic $\delta$-hyperbolic space.
\end{theorem}

It is also known that 

Since $\widehat{X}$ is not a proper metric space, we cannot define its boundary via equivalence classes of geodesic directions as we do for $\p_\K X$. Rather, we will define the Gromov-boundary of $\widehat{X}$ in terms of equivalence classes of sequences. For more of the following information, see \cite{BK02} or \cite{BH99}. Recall that a sequence $\{x_n\}$ \textit{Gromov-converges to infinity} if $\liminf_{i, j \rightarrow \infty}\left(x_{i} \cdot x_{j}\right)_{\ob}=\infty$ where $\left(x_{i} \cdot x_{j}\right)_{\ob}$ is the \textit{Gromov product}, defined as

$$\left(x_{i} \cdot x_{j}\right)_{\ob} =\frac{1}{2}\left(\hat{d}\left(\ob, x_{i}\right)+\hat{d}\left(\ob, x_{j}\right)-\hat{d}\left(x_{i}, x_{j}\right)\right).$$

 \noindent Two sequences $\{x_i\}$ and $\{y_j\}$ are said to be in the same equivalence class if 

$$\liminf_{i, j \rightarrow \infty}\left(x_{i} \cdot y_{j}\right)_{\ob}=\infty.$$

\noindent Thus, we can define the \textit{Gromov boundary} $\p\widehat{X}$ as the set of equivalence classes of sequences Gromov-converging to infinity. Given $s \in \p\widehat{X}$ and $r > 0$ the sets

$$U(s,r) = \left\{\left[\left\{y_i\right\}\right]\bigg| \liminf_{i, j \rightarrow \infty}\left(x_{i} \cdot y_{j}\right)_{\ob} \geq r  \text{ for some }\{x_i\} \in s\right\}.$$ form a basis for the standard topology on the Gromov boundary $\p\widehat{X}$. Furthermore, isometries of $\widehat X$ induce homeomorphisms of $\p \widehat{X}$ with respect to this topology.
\end{definition}

We desire to create a map $\p_\K X \longrightarrow \p\widehat{X}$ that sends $b^\infty \in \p_\K X$ to some $s \in \p \widehat{X}$. Ideally, the geodesic $b$ will project to some unbounded set in $\widehat{X}$, and increasing sequences of points on $b$ will Gromov-converge to infinity in $\p \widehat{X}$. However, it is not obvious that $b$ will always be unbounded in $\widehat{X}$. A potential problem is the strength of our sublinear function $\K$, as the following example will show. 

\begin{example}\label{example}
Let $\K=\sqrt{\cdot}$, i.e. the square root function, and consider the following space:

Put $X_{\sqrt{\cdot}}$ to be the subset of $\mathbb{R}^{2}$ between the square root function and the $x$-axis. Put $X_{i}=\left[i^{2}+\frac{1}{2},(i+1)^{2}-\frac{1}{2}\right] \times[0, \infty)$. Now denote $X=X_{\sqrt{\cdot}} \cup \bigcup_{i=1}^{\infty} X_{i}$. Notice that, when $b$ is the $x$-axis, $b$ is $\K$-contracting, and a $\K$-chain dual to $b$ is $\left\{h_{b, t_{i}}\right\}$ where $b\left(t_{i}\right)$ is the point $\left(i^{2}, 0\right)$. See Figure \ref{fig:3}, where $X$ is the blue-shaded space.

\begin{figure}[ht]
    \centering
\begin{tikzpicture}[scale=1]
 \pgfplotsset{
  width= 360,
  height= 170
  }
\begin{axis}[
    axis lines = middle,
    xmin=0, xmax=27,
    ymin=0, ymax=6,
    xtick={0,1,4,9,16,25},
    ytick={1,2,3,4,5}]
    \draw[teal!30, fill = teal!30  ] plot((0,0) -- (0.5,0) -- (0.5,6) -- (0,6) -- cycle;

   \draw[teal!30, fill = teal!30  ] plot((1.5,0) -- (3.5,0) -- (3.5,6) -- (1.5,6) -- cycle;
   \draw[teal!30, fill = teal!30  ] plot((4.5,0) -- (8.5,0) -- (8.5,6) -- (4.5,6) -- cycle;
   \draw[teal!30, fill = teal!30  ] plot((9.5,0) -- (15.5,0) -- (15.5,6) -- (9.5,6) -- cycle;
   \draw[teal!30, fill = teal!30  ] plot((16.5,0) -- (24.5,0) -- (24.5,6) -- (16.5,6) -- cycle;
    \draw[teal!30, fill = teal!30, ] plot((25.5,0) -- (27,0) -- (27,6) -- (25.5,6) -- cycle;

    \draw[dotted] (0.5,0) to (0.5, sqrt(0.5);
    \draw[dotted] (1.5,0) to (1.5, sqrt(1.5);
    \draw[dotted] (3.5,0) to (3.5, sqrt(3.5);
    \draw[dotted] (4.5,0) to (4.5, sqrt(4.5);
    \draw[dotted] (8.5,0) to (8.5, sqrt(8.5);
    \draw[dotted] (9.5,0) to (9.5, sqrt(9.5);
    \draw[dotted] (16.5,0) to (16.5, sqrt(16.5);
    \draw[dotted] (15.5,0) to (15.5, sqrt(15.5);
    \draw[dotted] (24.5,0) to (24.5, sqrt(24.5);
    \draw[dotted] (25.5,0) to (25.5, sqrt(25.5);

    \node[below] at (2.5,6) {\tiny $X_1$};
    \node[below] at (6.5,6) {\tiny $X_2$};
    \node[below] at (12.5,6) {\tiny $X_3$};
    \node[below] at (20.5,6) {\tiny $X_4$};
    
% Plot 1
\addplot [name path = A,
    -latex,
    domain = 0:27, dashed,
    samples = 1000] {sqrt(x)} ;
 
% Plot 2
\addplot [name path = B,
    -latex,
    domain = 0:27,] {0} ;
 
% Fill area between paths
\addplot [teal!30,] fill between [of = A and B, soft clip={domain=0:27}];

\end{axis}
\end{tikzpicture}
    \caption{The space $X=X_{\sqrt{\cdot}} \cup \bigcup_{i=1}^{\infty} X_{i} $}

    \label{fig:3}
\end{figure}
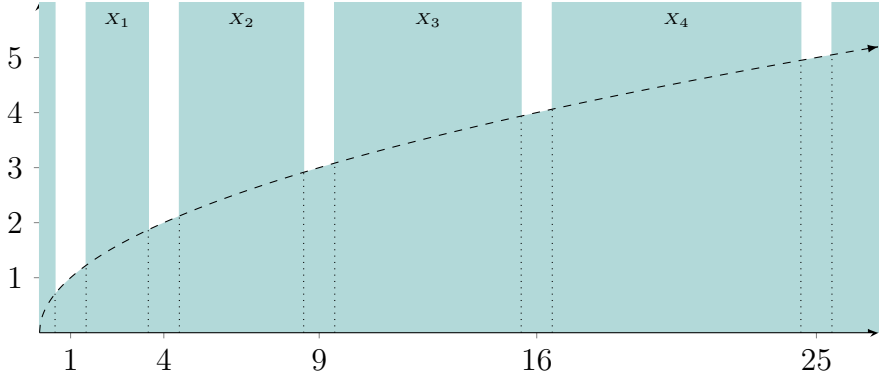

 %  \begin{center}
 %      \includegraphics[width=6in]{counter example.PNG}     
 % \end{center}

In this constructed space, any $L$-chain dual to $b$ must only have one curtain in each $X_{i}$. This is because any two curtains in $X_{i}$ dual to $b$ are not $L$-separated for any $L$. Though it is possible to find a larger $L$-chain such that its curtains are not dual to $b$, the remainder of Section \ref{Section 4} uses dual chains as a lower bound for distance. By observation we have that, for any $t$ with $t_{i} \leq t \leq t_{i+1}$, the largest $L$-chain dual to $[\ob, b(t)]$ is $2L$  if $ L \leq i$ and  $2i$ if $L>i$.

Thus, as $i\rightarrow \infty$, we have

$$
\lim_{i \to \infty} \hat{d}\left(\ob, b\left(t_{i}\right)\right)=\lim_{i \to \infty}\sum_{L=1}^{\infty} \frac{d_{L}\left(\ob, b\left(t_{i}\right)\right)}{L^{3}} \geq \sum_{L=1}^{\infty} \frac{2L}{L^{3}}.
$$

Notice how the right hand side of the inequality is a finite value. This argument is not sufficient to show that $diam(b)$ is unbounded in $\widehat{X}$. Thus, it is possible $b$ would not define a point in $\partial \widehat{X}$. It may still be possible to find collections of $L$-chains meeting $b$ that show $b$ has infinite diameter in $\widehat{X}$, but finding such $L$-chains will be tougher. Since most applications of sublinear functions involve a $\K$ such that $\K^4$ is sublinear, such as logarithmic functions, we can afford to impose this stronger condition for cleaner arguments. Thus, for the remainder of the paper, we assume $\K$ is any continuous, bijective, and monotonically increasing function such that $\K^4$ is sublinear.
\end{example}

 The following lemmas show that geodesics in $X$ project to unparameterized quasi-geodesics in $\widehat{X}$. Lemmas \ref{gluingLchain} and \ref{backtrack} assist in the proof of Lemma \ref{unparametized quasi geos}. Lemma \ref{unparametized quasi geos} shows us that projections behave like rough geodesics. It should be noted that Lemma \ref{unparametized quasi geos} shows the same result as \cite[Propostion 9.5]{PSZ22} but for when $\sum_{L=1}^\infty \lambda_L = \sum_{L=1}^\infty \frac{1}{L^3}$. Since the proof of Lemma \ref{unparametized quasi geos} was written before the the updated version of \cite{PSZ22}, we still include Lemma \ref{unparametized quasi geos} for completeness.

 \begin{lemma}[Lemma 2.13 in \cite{PSZ22}]\label{gluingLchain}
Suppose that $c=\left\{h_1, \ldots, h_n\right\}$ and $c^{\prime}=\left\{h_1^{\prime}, \ldots h_m^{\prime}\right\}$ are $L$ chains with $n>1$ and $m>L+1$, where $h_1^{\prime-}$ is the halfspace not containing $h_2^{\prime}$. If $h_1^{\prime-} \cap h_j \neq \varnothing$ for all $j$ and $h_n^+ \cap h_i^\prime \neq \emptyset$ for all $i$, then $c^{\prime \prime}=\left\{h_1, \ldots, h_{n-1}, h_{L+2}^{\prime}, \ldots, h_m^{\prime}\right\}$ is an $L$-chain of cardinality $n+m-L-2$.
\end{lemma}

 \begin{lemma}[Corollary 3.2 in \cite{PSZ22}] \label{backtrack}
If $b$ is a CAT(0) geodesic and $t_1<t_2<t_3$, then any $L$-chain c separating $b(t_2)$ from $\left\{b(t_1), b(t_3)\right\}$ has cardinality at most $L^{\prime}=1+\left\lfloor\frac{L}{2}\right\rfloor$.
\end{lemma}

\begin{lemma}\label{unparametized quasi geos}
Let $b :I \rightarrow X$ be a geodesic ray in $(X,d)$, then for any $t_1,t_2,t_3\in I$ such that $t_1<t_2<t_3$ with $b(t_1) = x, b(t_2)=y, b(t_3)=z$, we have $\hat{d}(x,z) \geq \hd(x,y) + \hd(x,z) - C$   where $C$ is a constant independent of $b$. \end{lemma} 

\begin{proof}
Let $L \in \mathbb{N}.$ Put $c$ and $c'$ the maximal $L-$chains realizing $ d_L(x,y) = 1 + |c|$ and $ d_L(y,z) = 1+ |c'|$. Assume $|c'| \geq \frac{3L}{2}+3$. By Lemma \ref{backtrack}, the maximum number of curtains in $c'$ that cross $[x,y]$ is $1 + \lfloor \frac{L}{2}\rfloor$. Denote $c''\subseteq c'$ the set with such curtains deleted so that $|c''| \geq |c'| - 1 - \lfloor \frac{L}{2}\rfloor $. Ordering $c'' = \{h_1, h_2 \ldots, h_n\}$, we denote $h_1^+$ as the halfspace containing $h_2$. Notice that $h_1^-$ must contain $[x,y]$, which means every curtain in $c$ meets $h_1^-$. Thus, by Lemma \ref{gluingLchain},  $d_L(x,z) \geq |c| + (|c'| - 1 - \lfloor \frac{L}{2}\rfloor) - (L+2) \geq d_L(x,y) + d_L(y,z) -  \frac{3L}{2}- 5$. Note that this inequality is also trivially true  if $|c'| \leq \frac{3L}{2}+3$.  Thus, we get that
\begin{align*}
    \hd(x,z) &= \displaystyle \sum_{L=1}^\infty \frac{d_L(x,z)}{L^3} \\
             &\geq \displaystyle \sum_{L=1}^\infty \frac{d_L(x,y) + d_L(y,z) -  \frac{3L}{2}- 5}{L^3} \\
             &= \hd(x,y) + \hd(y,z) - \displaystyle \sum_{L=1}^\infty( \frac{3}{2L^2} + \frac{5}{L^3}).
\end{align*}

Setting $C = \displaystyle \sum_{L=1}^\infty( \frac{3}{2L^2} + \frac{5}{L^3})$ gives the desired result.
\end{proof}
\begin{remark}
    Lemma \ref{unparametized quasi geos} shows a ``coarse reverse triangle inequality" for geodesics projecting into $\widehat{X}$. Looking at the the definition of an unparameterized quasi geodesic (see \cite[Section 2.1]{MMS12}), one can use Lemma \ref{unparametized quasi geos} to parameterize a projected geodesic into a quasi-geodesic in $\widehat{X}$.
\end{remark}

In order the show that any $\K$-contracting geodesic defines a point in $\p\hX$, we must also show that the geodesic will have infinite diameter with respect to $\hd$.

\begin{lemma}\label{unbnd}
Let $X$ be a CAT(0) space and $b$ be a $\K$-contracting geodesic ray based at $\ob$ where $\K$ is a sublinear function such that $\K^4$ is sublinear.  Then $b$ has infinite diameter with respect to the $\hd$ metric.
\end{lemma}

\begin{proof}
Denote $\{h_i\}$ the $\K$-chain dual to $b$ and denote $b(t_i)$ the center of the poles of each $h_i$. Put $C>0$ as the excursion constant. We know that $\lfloor C\K(t_1)\rfloor = n$ for some $n \in \mathbb{N}.$ Thus, for any $i>1$, there exists some  $m\geq n$ such that $\lceil C\K(t_{i})\rceil = m$. Since, by definition of a $\K$-chain, $t_{j+1} - t_j \leq C\K(t_{j+1})$, we have  \begin{align*}
    t_{i} - t_1 &=\displaystyle\sum_{j=1}^{i} (t_{j+1} - t_{j}) \\
                  &\leq \displaystyle\sum_{j=1}^{i} C\K(t_{j+1}) \\
                  &\leq  \displaystyle\sum_{j=1}^{i} C\K(t_{i})\\
                  &=i C\K(t_{i}) \\
                  &\leq i m.
\end{align*}
Thus, we have $i \geq \frac{t_{i} - t_1}{m}$. Notice, since $\lceil C\K(t_{i})\rceil = m$, this implies $C\K(t_{i}) + 1\geq m$. Now, we claim $d_m(\mathfrak{o}, b(t_i)) \geq i-1 \geq \displaystyle \frac{t_i - t_1}{m} -1$. Indeed,  $[\ob, b(t_i)]$ crosses curtains $\{h_1, h_2, \ldots, h_{i-1}\}$, and since each $h_{j-1} $ and $h_j$ are $C\K(t_j)$-separated, they are also $m$-separated since $C\K(t_j) \leq m$ for all $j \leq i-1$. Thus, it follows that \[ \frac{d_m(\mathfrak{o}, b(t_i))}{m^3} \geq \displaystyle \frac{\frac{t_i - t_1}{m}-1}{m^3} \geq \frac{t_i - t_1}{m^4}-1 \geq \displaystyle \frac{t_i - t_1}{(C\K(t_{i}) + 1)^4} - 1.\]
 
 By assumption, $\K^4$ is a sublinear function. Hence, as $i$ grows to infinity, the right hand side of the above inequality grows to infinity.  Since $\hd(\ob, b(t_i)) \geq \frac{d_m(\mathfrak{o}, b(t_i))}{m^3}$, we conclude that \[\lim_{i\rightarrow \infty} \hd(\ob, b(t_i))  = \infty.\]
 Hence, $b$ is unbounded in $\hX$.
 \end{proof}
 
 \begin{remark}\label{d hat is faster than k}
  Notice, in the above proof, we showed that for each $t_i$, we can find an $m \in \mathbb{N}$ such that
$$
\frac{d_{m}\left(\ob, b\left(t_{i}\right)\right)}{m^{3}} \geq \frac{t_{i}-t_{1}}{\left(c \K\left(t_{i}\right)+1\right)^{4}} -1.
$$

  This inequality also shows the possibility of $\K$-contracting rays having a \textit{$\K$-persistent shadow} as defined in \cite{DZ22}. We show a notion of equivalence between a ray being $\K$-contracting and having a $\K$-persistent shadow in Section \ref{Section 5}, and we postpone further conversation to that section.

\end{remark}

\noindent 4.2 \textbf{ Creating an injective map $ \partial_{\K}X \longrightarrow \p \widehat{X}$.} \vspace{.25cm}

Assume $X$ is a proper CAT(0) space with $\widehat X$ its curtain model and let $\K$ be a sublinear function such that $\K^4$ is sublinear. We have that $\K$-contracting geodesics in $X$ will be unbounded quasi-geodesics in $\widehat{X}$. Since each $b^\infty \in \p_\K X$ has only one geodesic emanating from $\ob$ in its equivalence class, we define $\varphi$ by

\begin{align*}
    \varphi: \p_\K X & \longrightarrow \p \widehat{X} \\
                    b^\infty &\longmapsto \big[ \{b(n)\}_{n \in \mathbb{N}} \big ]
\end{align*}

 Notably, the map is well defined since $b^\infty$ only has one geodesic emanating from $\ob$ (namely, $b$). Also, given any increasing sequence $\{t_i\}$ with $t_i \rightarrow \infty$, we get that $\{b(t_i)\}$ Gromov-converges to infinity and  $[\{b(t_i)\}] = [ \{b(n)\}]$. This fact is useful because we can choose any increasing sequence $\{b(t_i)\} \subset b$ that Gromov-converges to infinity to represent $\varphi(b^\infty)$ in our proofs. We now work to make $\varphi$ injective.

 \begin{lemma}\label{welldefined}
 Let $b_{1}$ be a $\K$-contracting geodesic ray. Denote $\left\{h_{i}\right\}$ the $\K$-chain dual to $b_1$ with center of poles $b_1\left(t_{i}\right)$. If a geodesic $b_{2}$ meets infinitely many of $\left\{h_{i}\right\}$, then $b_{2}$ is in a $\K$-neighborhood of $b_{1}$.
 \end{lemma}

 \begin{figure}[ht]
    \centering
   \begin{tikzpicture}[scale=1]

    % \draw[red]  (2.8,-3) .. controls(3.1,-3)  and (3,.5)   .. (2.8,2.5);
    % \draw[red] (3.8,-3) .. controls (3.5, -3) and (3.6, .5) .. (3.8,2.5);

    \fill[red, opacity=0.2] (2.8,-3) .. controls(3.1,-3)  and (3,.5)   .. (2.8,1.5) -- (3.8,1.5).. controls (3.6, .5)  and (3.5, -3) .. (3.8,-3) -- cycle;

    \fill[red, opacity=0.2] (2.8+4,-3) .. controls(3.1+4,-3)  and (3+4,.5)   .. (2.8+4,1.5) -- (3.8+4,1.5).. controls (3.6+4, .5)  and (3.5+4, -3) .. (3.8+4,-3) -- cycle;

    % \draw[red!50, fill = red!50, opacity=0.9  ] plot((3,-3) -- (3.6,-3) -- (3.6,1.5) -- (3,1.5) -- cycle;

    % \draw[red!50, fill = red!50, opacity=0.9  ] plot((7,-3) -- (7.6,-3) -- (7.6,1.5) -- (7,1.5) -- cycle;
    %\draw[help lines] (0,0) grid (10,6);
    %\node [below] at (0,0) {$\null$};
    %\node [below] at (10,6) {$\null$};
    \draw [blue,thick, ->](1,-2) to (10,-2);
    \node [below, blue] at (10,-2) {$b_1$};

    \draw [->, thick] (1,1) .. controls (4,0) and (7,0) .. (10,1);
    \node[below] at (10,1) {$b_2$};

    \node[below] at (3.3, .44) {$q_i$};
    \draw[fill] (3.3,.44) circle [radius=0.05];

    \node[below] at (7.3, .38) {$q_{i+1}$};
    \draw[fill] (7.3,.38) circle [radius=0.05];

    \node[below] at (3.3, -2) {$b_1(t_i)$};
    \draw[fill] (3.3,-2) circle [radius=0.05];

    \node[below] at (7.3, -2) {$b_1(t_{i+1})$};
    \draw[fill] (7.3,-2) circle [radius=0.05];

    \node[below] at (5, .27) {$p$};
    \draw[fill] (5,.27) circle [radius=0.05];

    \node[red, above] at (3.3,1.5) {$h_i$};
    \node[red, above] at (7.3,1.5) {$h_{i+1}$};

   % Simple brace
\draw [decorate,
    decoration = {brace}] (2.95,-2) --  (2.95,.44);
    \node [left] at (2.95,-.78) {$D\kappa(t_{i}) \geq$};
\draw [decorate,
    decoration = {brace}] (7.65,.38) --  (7.65,-2);
    \node [right] at (7.65,-.81) {$\leq D\kappa(t_{i+1})$};
\draw [decorate,
    decoration = {brace}]  (7.3,-2.5) -- (3.3,-2.5);
    \node [below] at (5.3,-2.5) {$\leq D'\kappa(t_{i+1})$};

    \end{tikzpicture}
    \caption{A picture for the proof of Proposition \ref{welldefined}. We show that $b_2$ is in a $\K$-neighborhood of $b_1$.}
    \label{fig:bounding A_2}
\end{figure}
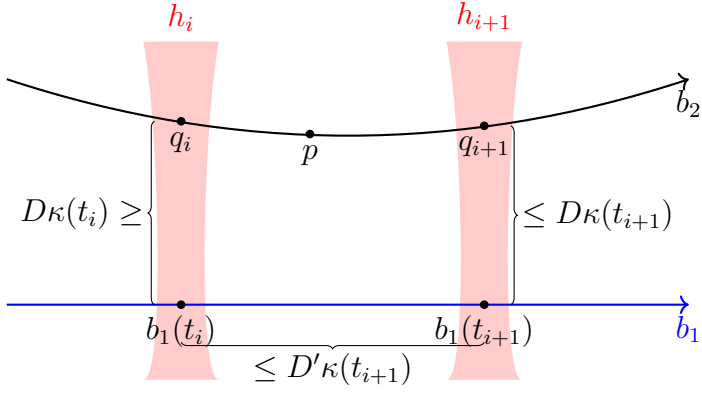
 
 \begin{proof}
 Since the $\left\{h_{i}\right\}$ are ordered, then $b_{2}$ crossing $h_{m}$ and $h_{n}$ implies that $b_{2}$ also crosses all curtains of $\left\{h_{i}\right\}$ between $h_{m}$ and $h_{n}$. Thus, $b_{2}$ must cross all but finitely many $h_{i}$. Removing such finitely many curtains will still result in a $\K$-chain dual to $b_1$, so we assume $b_{2}$ meets all of $\left\{h_{i}\right\}$ at points $q_{i} \in b_{2}$ such that $\pi_{b_1}(q_i) = b_1(t_i)$. Now, Lemma \ref{closeness} tells us that $d(q_{i}, b_{1}\left(t_{i}\right)) \leq D \K\left(t_{i}\right)$ for some $D>0$ depending only on the contracting constant $C$. Also, by definition of a $ \K$-chain and Lemma \ref{3.2}, $d(b_{1}\left(t_{i}\right), b_{1}\left(t_{i+1}\right)) \leq D^{\prime} \K\left(t_{i}\right)$ for some $D^{\prime}>0$ depending only on $C$. Thus, for any $p \in b_{2}$ between $q_{i}$ and $q_{i+1}$, we have

$$
\begin{aligned}
d\left(p, q_{i}\right) & \leq d\left(q_{i}, q_{i+1}\right) \\
& \leq d\left(q_{i}, b_1(t_{i}))+d\left(b_1(t_{i}), b_1(t_{i+1})\right )+d\left(b_1(t_{i+1}), q_{i+1}\right)\right.\\
& \leq D \K\left(t_{i}\right)+D^{\prime} \K\left(t_{i}\right)+D \K\left(t_{i+1}\right) \\
& \leq D^{\prime \prime} \K\left(t_{i}\right)
\end{aligned}
$$

for some $D^{\prime \prime}>0$ by Lemma \ref{3.2}. This gives, 

$$
\begin{aligned}
d\left(p, b_{1}\right) & \leq d\left(p, q_{i}\right)+d\left(q_{i}, b_1\left(t_{i}\right)\right) \\
& \leq\left(D^{\prime \prime}+ D\right) \K\left(t_{i}\right) \\
& \leq\left(D^{\prime \prime}+ D\right) \K(p).
\end{aligned}
$$

This is true for any $p \in$  $b_{2}$ between some $q_{i}$ and $q_{i+1}$, which gives us that $b_{2}$ is in a $\K$-neighborhood of $b_{1}$. See Figure \ref{fig:bounding A_2}.
 \end{proof}

 \begin{remark}\label{contrapositive}
 Note that, by the contrapositive, if $b_1$ and $b_2$ are $\K$-contracting rays that are not in the same $\K$-equivalence class, then each only crosses up to finitely many curtains in the other's  $\K$-chain.
 \end{remark}

\begin{proposition}[Injectivity of $\varphi$]\label{injective} 
Let $b_1^\infty,b_2^\infty \in \partial_{\K} X$ and let $b_1, b_2$ be the corresponding $\K$-contracting geodesic rays based at $\ob$. If $b_1^\infty \neq b_2^\infty$, then $\varphi(b_{1}^\infty) \neq \varphi(b_{2}^\infty)$ in $\partial \widehat{X}$.
\end{proposition}

\begin{proof}
Since $b_{1}$ and $b_{2}$ are both $\K$-contracting rays, they both have $\K$-chains dual to them. Denote $\left\{h_{i}\right\}$ and $\left\{h_{j}^{\prime}\right\}$ as the $\K$-chains dual to $b_{1}$ and $b_{2}$ with centers of poles $\left\{b_{1}\left(t_{i}\right)\right\}$ and $\left\{b_{2}\left(t_{j}^{\prime}\right)\right\}$, respectively. By Remark \ref{contrapositive}, $b_{1}$ and $b_{2}$ only cross finitely many of the other's $\K$-chain. Put $h_{m}$ to be the last curtain of $\left\{h_{i}\right\}$ that $b_{2}$ crosses and $h_{n}^{\prime}$ to be the last curtain of $\left\{h_{j}\right\}$ that $b_{1}$ crosses. Consider the sequence $\left\{x_{i}\right\}$ where $x_{i}=b_1\left(t_{i}+1\right)$ and $\left\{y_{j}\right\}$ where $y_{j}=b_{2}\left(t_{j}^{\prime}+1\right)$.

Now, $h_{m+1}$ and $h_{m+2}$ are $L_{1}$-separated for $L_{1}= \lceil \K(t_{m+2}) \rceil$ and, for $i \geq m+2$, separate $x_{i}$ from ${b}_{2}$. Similarly, $h_{n + 1}^{\prime}$ and $h_{n+2}^{\prime}$ are $L_{2}$-separated for $L_{2} = \lceil \K(t'_{n+2}) \rceil$ and for $j \geqslant n+2$, separate $y_{j}$ from $b_{1}$. See the Figure \ref{fig:injectivity}.

\begin{figure}[ht]
    \centering
    \begin{tikzpicture}[scale=1]
    %\draw[help lines] (0,0) grid (10,6);
    %\node [below] at (0,0) {$\null$};
    %\node [below] at (10,6) {$\null$};
    
    \draw [blue,thick, ->](1,-2.03) .. controls (5,-2.3) and (8,-2.3) .. (10.2,-5.5);
    \node [below, blue] at (10.2,-5.5) {$b_1$};
    \draw [red,thick, ->](1,-1.97) .. controls (5,-1.7) and (8,-1.7) .. (10.2,1.5);
    \node [above, red] at (10.2,1.5) {$b_2$};
    \draw[fill] (1,-2) circle [radius=0.05];
    \node[below] at (1,-2) {$\ob$};
    \draw[red!50, fill = red!50, opacity=0.3  ] plot((2.1,-3) -- (2.25,-2.98) -- (2,-.98-.1) -- (1.85,-1-.1) -- cycle;
     \draw[red!50, fill = red!50, opacity=0.3  ] plot((2.9,-3) -- (3.05,-2.98) -- (2.7,-.98-.1) -- (2.55,-1-.1) -- cycle;
    \draw[red!50, fill = red!50, opacity=0.3 ] plot((3.7,-3) -- (3.85,-2.98) -- (2.7+.7,-.98-.1) -- (2.55+.7,-1-.1) -- cycle;

    \draw[blue!50, fill = blue!50, opacity=0.3  ] plot((1.9,-3+.1) -- (2.05,-3.02+.1) -- (2.3,-.98) -- (2.15,-.96) -- cycle;

    \draw[blue!50, fill = blue!50, opacity=0.3  ] plot((1.9+.65,-3+.1) -- (2.05+.65,-3.02+.1) -- (2.3+.75,-.98) -- (2.15+.75,-.96) -- cycle;

    \draw[blue!50, fill = blue!50, opacity=0.3  ] plot((1.9+.65+.7,-3+.1) -- (2.05+.65+.7,-3.02+.1) -- (2.3+.75+.7,-.98) -- (2.15+.75+.7,-.96) -- cycle;

    \draw[fill] (4.3+.3,-2) circle [radius=0.02];
    \draw[fill] (4.5+.3,-2) circle [radius=0.02];
    \draw[fill] (4.7+.3,-2) circle [radius=0.02];
    \draw[fill] (4.9+.3,-2) circle [radius=0.02];

    \draw[red!50, fill = red!50, opacity=0.3  ] plot((7,-4) -- (7.15,-3.95) -- (6,0) -- (5.85,-.05) -- cycle;
    \draw[red!50, fill = red!50, opacity=0.3  ] plot((10.85-.2,-5-.2) -- (11-.2,-4.95-.15) -- (7.5-.2,.6-.15) -- (7.35-.2,.55-.2) -- cycle;
    \draw[red!50, fill = red!50, opacity=0.3  ] plot((12.35-.2,-4.5-.2) -- (12.5-.2,-4.45-.15) -- (9-.3,1.15) -- (8.85-.3,1.15-.1) -- cycle;

    \draw[blue!50, fill = blue!50, opacity=0.3  ] plot((6.2,-4) -- (6.35,-4.02) -- (6.85,0) -- (6.7,-.02) -- cycle;
    \draw[blue!50, fill = blue!50, opacity=0.3  ] plot((7.15,-4.35) -- (7.3,-4.45) -- (10.85,1.2) -- (10.7,1.3) -- cycle;
    \draw[blue!50, fill = blue!50, opacity=0.3  ] plot((8.55,-5.05) -- (8.7,-5.15) -- (12.3,.6) -- (12.15,.7) -- cycle;

    \node[below, blue] at (6.275, -4.01) {\small$h_m$};
    \node[below, blue] at (7.225, -4.4) {\small$h_{m+1}$};
    \node[below, blue] at (8.625, -5.1) {\small$h_{m+2}$};

    \node[above, red] at (5.925, -.025) {\small$h'_n$};
    \node[above, red] at (7.225, .4) {\small$h'_{n+1}$};
    \node[above, red] at (8.625, 1.1) {\small$h'_{n+2}$};

   % Simple brace
\draw [decorate,blue, 
    decoration = {brace}] (10.75,1.4) --  (12.35,.7);
    \node [blue, above] at (12.55,1.05) {$L_1-$separated};

\draw [decorate,red, 
    decoration = {brace}] (12.35,-4.7) -- (10.7,-5.3);
    \node [red, below] at (12.525,-5) {$L_2-$separated};

    \draw[fill] (9.4,0.5) circle [radius=0.05];
    \draw[fill] (9.4,-4.5) circle [radius=0.05];
    \node[right] at (9.4, 0.5) {\small $y_{n+2}$};
    \node[right] at (9.4, -4.5) {\small $x_{m+2}$};

    \end{tikzpicture}
    \caption{Two $L_1-$separated curtains separate $x_{m+2}$ from $b_2$. Similarly, two $L_2-$separated curtains separate $y_{n+2}$ from $b_1$.   }
    \label{fig:injectivity}
\end{figure}
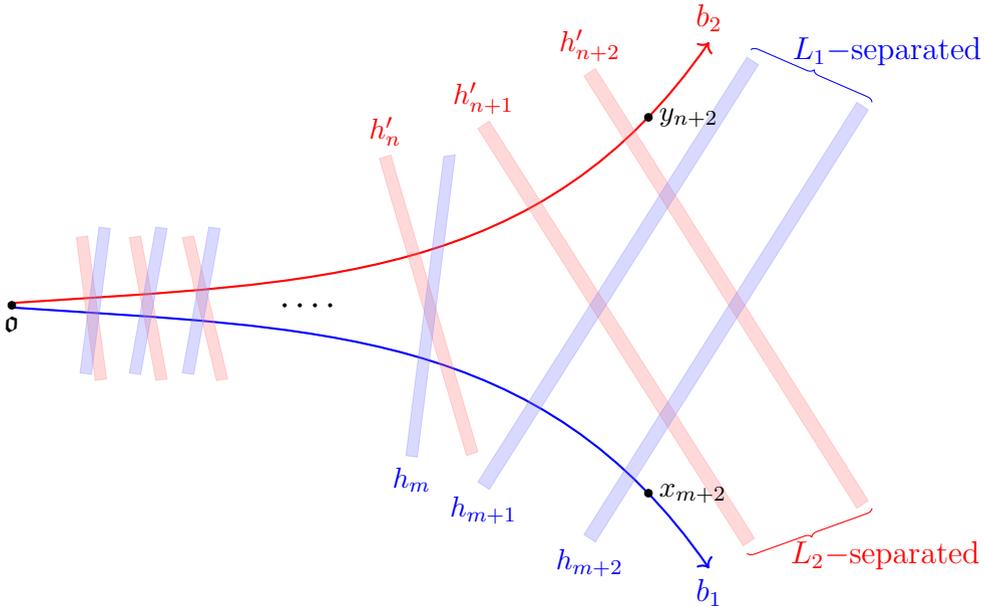

We now investigate $d_{L}\left(x_{i}, y_{j}\right)$ for $i>m+2, j>n+2$. Let $c$ be a maximal $L$-chain that realizes $d_{L}\left(x_{i}, x_{m+2}\right)=1+|c|$. Due to where $h_{m+1}$ and $h_{m+2}$ are placed, at least $|c|-L_1$  curtains in $c$ will not intersect $b_{2}$. Let $c^{\prime}$ be a maximal chain realizing $d_{L}\left(y_{j}, y_{n+2}\right)=1+\left|c^{\prime}\right|$. By a similar argument involving $h^\prime_{n+1}$ and $h^\prime_{n+2}$, we get that at least $|c^\prime| - L_2$ curtains in $c^\prime$ will not intersect $b_1$. By gluing these two chains together via Lemma \ref{gluingLchain}, we obtain an $L$-chain of length at least $|c|+\left|c^{\prime}\right|-L_{1}-L_2-L-2$ that separates $x_{i}$ and $y_{j}$. Thus,

$$
d_{L}\left(x_{i}, y_{j}\right) \geq d_{L}\left(x_{i}, x_{m+2}\right)+d_{L}\left(y_{j}, y_{n+2}\right)-L_{1} - L_2 -L-4.
$$

Hence,

\begin{align*}
 \left(x_{i} \cdot y_{j}\right)_{\ob}&=\frac{1}{2}\left(\hat{d}\left(\ob, x_{i}\right)+\hat{d}\left(\ob, y_{j}\right)-\hat{d}\left(x_{i}, y_{j}\right)\right) \\
& =\frac{1}{2}\left(\hat{d}\left(\ob, x_{i}\right)+\hat{d}\left(\ob, y_{j}\right)-\sum_{L=1}^{\infty} \frac{d_{L}\left(x_{i}, y_{j}\right)}{L^{3}}\right) \\
& \leq \frac{1}{2}\left(\hat{d}\left(\ob, x_{i}\right)+\hat{d}\left(\ob, y_{j}\right)-\sum_{L=1}^{\infty} \frac{d_{L}\left(x_{i}, x_{m+2}\right)+d_{2}\left(y_{i}, y_{n+2}\right)-L_{1} - L_2-L-4}{L^{3}}\right) \\
& =\frac{1}{2}\left(\hat{d}\left(\ob, x_{i}\right)+\hat{d}\left(\ob, y_{j}\right)-\hat{d}\left(x_{i}, x_{m+2}\right)-\hat{d}\left(y_{i}, y_{n+2}\right )+\sum_{L=1}^{\infty} \frac{L_{1}+L_2+L+4}{L^{3}}\right)  \\
& \leq \frac{1}{2}\left(\hat{d}\left(\ob, x_{m+2}\right)+\hat{d}\left(\ob, y_{n+2}\right)+\sum_{L=1}^{\infty} \frac{L_{1}+L_2+L +4}{L^{3}}\right).
\end{align*}

Where the last inequality is due to the triangle inequality. Thus, for any $i \geq m+2$ and $j \geq n+2$, we have that $\left(x_{i}\cdot y_{j}\right)_{\ob}$ is bounded by the above constant. This gives $$\liminf_{i, j \rightarrow \infty}\left(x_{i} \cdot y_{j}\right)_{\ob}<\infty,$$ so $\left\{x_{i}\right\}$ and $\left\{y_{j}\right\}$ represent different equivalence classes in $\partial \widehat{X}$. This implies the same for $\varphi(b_{1}^\infty)$ and $\varphi(b_{2}^\infty)$.

\end{proof}

\noindent 4.3 \textbf{ Continuity of $ \varphi: \p_\K X  \longrightarrow \p \widehat{X}$} \vspace{.25cm}

We now review the topologies of both $\p_\K X$ and $\p \widehat{X}$ in preparation for showing continuity of $ \varphi$. Recall that we defined the topology of $\p \widehat{X}$ by the basic open sets $$U(s,r) = \left\{\left[\left\{y_i\right\}\right]\bigg| \liminf_{i, j \rightarrow \infty}\left(x_{i} \cdot y_{j}\right)_{\ob} \geq r  \text{ for some }\{x_i\} \in s\right\}$$ where $s \in \p \widehat{X}$ and $r > 0$. Showing continuity will involve using intersections of cone topology bases (Definition \ref{Cone }) and Curtain topology sets (Definition \ref{curtain top}).

\begin{definition}[Visual boundary, cone topology as a subspace topology] \label{Cone }The \textit{visual boundary} of a proper CAT(0) space X, denoted $\p X$ as a set, is the set of all geodesic rays emanating from $\ob$. The \textit{cone topology} of $\p X$ is generated by the basic open sets $$
 V_{R, \epsilon}(\xi):=\left\{\eta \in \partial X \mid d\left(\xi(R), \eta(R)\right)<\epsilon\right\}.
 $$
 When $\p X$ is equipped with the cone topology, we denote it $\p_\infty X$. For convenience, we denote such sets in the subspace $\p_\K X$ as $$
 U_{R, \epsilon}(\xi):=V_{R, \epsilon}(\xi) \cap \partial_\kappa X .
 $$ This defines a subspace topology on $\p_\K X.$
\end{definition}

The following topology on $\p_\K X$ is a useful topology to be able leverage curtain machinery in the same way hyperplanes were leveraged in the analogous topology defined in \cite{IMZ21}.

\begin{definition}[Curtain Topology Sets] \label{curtain top} Let $b^\infty \in\partial_\kappa X$ be a geodesic ray emanating from $\ob$. We say a curtain $h$ \textit{separates} $\ob$ from $a^\infty$ if there exists a $T>0$ such that $h$ separates $\ob$ from $a_{[T,\infty]}$. For each curtain $h$ dual to $b$, we define
$$
U_h\left(b^\infty\right)=\left\{a^\infty \in \partial X: h \text { separates } \mathfrak{o}=a(0) \text { from } a^\infty\right\} \cap \partial_\kappa X.
$$

We define the \textit{curtain topology on $\p_\K X$} as follows: a set $O \subset \p_\K X$ is open if for each $b^\infty \in O$, there exists a curtain $h$ dual to $b$ such that $U_h(b^\infty) \subset O$. This yields a topology on $\p _\K X$. Note, such a topology is actually the subspace topology of the curtain topology on all of $\p X$. However, for convenience of notation, we define $U_h(b^\infty)$ to be intersecting with $\p_\K X$ as above since we are only interested in $\p_\K X$ for this paper.
\end{definition}

\begin{lemma} \label{hyp open} The cone topology and the curtain topology agree on $\p_\K X$.

\end{lemma}

\begin{proof}

Showing the curtain topology is coarser than the cone topology is given in the proof of Theorem 8.8 in \cite{PSZ22}. The reverse direction is a recreation of the proof of Theorem 4.2 in \cite{IMZ21}, but in the curtain setting.

Consider some $U_{R, \epsilon}(b^\infty ) $ for some $\K$-contracting geodesic $b$ and $R, \epsilon >0$. Then there is an infinite $\K$-chain $\{h_i\}$ dual to $b$. Denote the center of poles of each $h_i$ as $b(t_i)$. Put $D$ as the constant in Lemma \ref{closeness} (note that $D$ only depends on the excursion constant of $b$). Fix $m$ large enough so that $t_m > R$ and $2D\K(t_m) \leq \frac{t_m}{R}\epsilon$. We claim $U_{h_{m+1}}(b^\infty) \subset U_{R,\epsilon}(b^\infty).$ Indeed,  let $a^\infty \in U_{h_{m+1}}(b^\infty) $, so $a$ is a geodesic emanating from $\ob$ that crosses $h_{m+1}$. By Lemma \ref{closeness}, there exists an $s >0$ such that $d(a(s),b(t_m)) \leq D\K(t_m)$. Note that, since $a$ and $b$ are both geodesics emanating from $\ob$, we get that $|s-t_m| \leq D\K(t_m) $ Thus,
$$  d(a(t_m),b(t_m)) \leq d(a(t_m),a(s)) + d(a(s),b(t_m)) \leq 2D\K(t_m) \leq  \frac{t_m}{R}\epsilon  $$

% d\left(a\left(t_m\right),a\left(s\right)\right) + d\left(a\left(s\right),b\left(t_m\right)\right) 

Due to convexity of the CAT(0) metric (as shown in \cite[II.2.1]{BH99} and its proof), $d(a(R), b(R)) \leq \epsilon$. Thus, $a^\infty \in  U_{R, \epsilon}(b^\infty ).$

\end{proof}
\begin{theorem}\label{Continuous}
The map $\varphi: \partial_{\K}X \longrightarrow \p \widehat{X}$ is a well-defined, injective, and  continuous map when $\partial_{\K}X$ is endowed with subspace topology of the cone topology.
\end{theorem}

\begin{proof} Proposition \ref{injective} shows that $\varphi$ is injective. What is left is to show continuity. Let $U(s, r)$ be an open set in $\p \widehat{X}$. If its preimage is nonempty, we can assume that there exists a $b^\infty \in \partial_{\K} X$ such that $ \varphi(b^\infty)=s$ and it suffices to show that there exists a $U_h(b^\infty)  \subset \varphi^{-1}(U(s, r)) $ for some curtain $h$ dual to $b$.

Let $r^{\prime}>0$ and $\epsilon>0$. Since $b$ is  a $ \K$-contracting geodesic, denote $\left\{h_{i}\right\}$ the $\K$-chain dual to $b$ with $b\left(t_{i}\right)$ the center of the pole of each $h_i$. Since $b$ is unbounded in $\widehat{X}$, there exists an $i=i_{ r^{\prime}}$ such that $\hat{d}\left(\ob, b\left(t_{i_{r^\prime}}\right)\right)>r^\prime$. Consider the set $U_{h_{i_{ r^\prime}}}(b^\infty) $ and the open set $U_{R, \epsilon}(b^\infty)$ where $R>t_{i_{r^\prime}}$. Since $U_{R, \epsilon}(b^\infty)$ is open, there exists a curtain $k$ dual to $b$ such that $U_k(b^\infty) \subseteq U_{R, \epsilon}(b^\infty)$ by Lemma \ref{hyp open} . Without loss of generality, the pole of the curtain $k$ is farther distance away from $\ob$ in the CAT(0) metric than the pole of $h_{t_{i_{r^\prime}}}$, so $U_k(b^\infty) \subset U_{h_{t_{i_{r^\prime}}}}(b^\infty)$.  We claim

$$
\begin{aligned}
&U_k(b^\infty) \subseteq \varphi^{-1}(U(s, r)). 
\end{aligned}
$$

Indeed, let $a^\infty \in U_k(b^\infty)$. So $a$ is the geodesic in $a^\infty$ emanating from $\ob$. Then, $a$ crosses the first $i_{r^\prime}$  curtains of $\left\{h_{i}\right\}$ and also $d(a(t), b)<\epsilon$ for all $t \leq R$. See Figure \ref{fig:continuity}. Now, consider any unbounded and increasing sequence $\left\{x_{i}\right\} \subseteq b$ and $\left\{y_{j}\right\} \subseteq a$ with the added condition that, for some $n$, $x_{n}=b\left(t_{i_{r^\prime}}{ }+\frac{1}{2}\right)$ and $y_{n}$ such that $d\left(x_{n}, y_{n}\right)<\epsilon$. What's left to show is that $\left(x_{m} \cdot y_{m^{\prime}}\right)_{\ob} \geq r$ for all $m, m^{\prime}>n$. By Lemma \ref{unparametized quasi geos}, we have there exists a constant $C$ such that, 

\begin{figure}
    \centering
    \begin{tikzpicture}[scale=1]    
    \draw [thick, ->](1,-2) .. controls (5,-2) and (8,-2) .. (11,-2);
    \node [below] at (11,-2) {$b$};
    \draw [blue,thick, ->](1,-1.97) .. controls (5,-1.8) and (8,-1.7) .. (8.5,1.5);
    \node [above, blue] at (8.5,1.5) {$a$};
    \draw[fill] (1,-2) circle [radius=0.05];
    \node[below] at (1,-2) {$\ob$};

     % \draw[red]  (1.8,-3) .. controls(2.1,-3)  and (2,-.5)   .. (1.8,-.5);
     % \draw[red]  (2.5,-3) .. controls(2.2,-3)  and (2.3,-.5)   .. (2.5,-.5);

    \fill[red, opacity=0.2] (1.9,-3) .. controls(2.1,-3)  and (2,-.5)   .. (1.9,-.5) -- (2.4,-.5).. controls (2.3,-.5)  and (2.2,-3) .. (2.4,-3) -- cycle;

    \fill[red, opacity=0.2] (1.9+.8,-3) .. controls(2.1+.8,-3)  and (2+.8,-.5)   .. (1.9+.8,-.5) -- (2.4+.8,-.5).. controls (2.3+.8,-.5)  and (2.2+.8,-3) .. (2.4+.8,-3) -- cycle;

    \fill[red, opacity=0.2] (1.9+1.8,-3) .. controls(2.1+1.8,-3)  and (2+1.8,-.5)   .. (1.9+1.8,-.5) -- (2.4+1.8,-.5).. controls (2.3+1.8,-.5)  and (2.2+1.8,-3) .. (2.4+1.8,-3) -- cycle;

    \fill[red, opacity=0.2] (1.9+5,-3) .. controls(2.1+5,-3)  and (2+5,-.5)   .. (1.9+5,.5) -- (2.4+5,.5).. controls (2.3+5,-.5)  and (2.2+5,-3) .. (2.4+5,-3) -- cycle;

    \fill[red, opacity=0.2] (1.9+7.5,-3) .. controls(2.1+7.5,-3)  and (2+7.5,-.5)   .. (1.9+7.5,2) -- (2.4+7.5,2).. controls (2.3+7.5,-.5)  and (2.2+7.5,-3) .. (2.4+7.5,-3) -- cycle;

    % \draw[blue!50, fill = blue!50, opacity=0.3  ] plot((2,-3) -- (2.3,-3) -- (2.3,-.5) -- (2,-.5) -- cycle;

    % \draw[blue!50, fill = blue!50, opacity=0.3  ] plot((2+.8,-3) -- (2.3+.8,-3) -- (2.3+.8,-.5) -- (2+.8,-.5) -- cycle;

    % \draw[blue!50, fill = blue!50, opacity=0.3  ] plot((2+1.8,-3) -- (2.3+1.8,-3) -- (2.3+1.8,-.5) -- (2+1.8,-.5) -- cycle;

 % \draw[blue!50, fill = blue!50, opacity=0.3  ] plot((2+5,-3) -- (2.3+5,-3) -- (2.3+5,.5) -- (2+5,.5) -- cycle;

 % \draw[blue!50, fill = blue!50, opacity=0.3  ] plot((2+7.5,-3) -- (2.3+7.5,-3) -- (2.3+7.5,2) -- (2+7.5,2) -- cycle;

    \draw[fill] (4.3+.8,-2+.15) circle [radius=0.02];
    \draw[fill] (4.5+.8,-2+.15) circle [radius=0.02];
    \draw[fill] (4.7+.8,-2+.15) circle [radius=0.02];
    \draw[fill] (4.9+.8,-2+.15) circle [radius=0.02];

    \draw[fill] (7.3, -2) circle [radius=0.05];
    \draw[fill] (7.3, -.73) circle [radius=0.05];
    \node[below] at (7.3, -2) {$x_n$};
    \node[above] at (7.3, -.73) {$y_n$};
    \draw[dotted] (7.3,-2) to (7.3,-.73);

    \draw[fill] (10.4, -2) circle [radius=0.05];
    \draw[fill] (8.08, .2) circle [radius=0.05];
    \node[above] at (7.9, .2) {$y_m$};
    \node[below] at (10.4, -2) {$x_m$};

    \node[below, red] at (7.15, -3) {$h_{i_{r'}}$};

    \node[right] at (7.3, -1.315) {\small $<\epsilon$};

    \end{tikzpicture}
    \caption{Picture of $b$ crossing $h_{i_{r'}}$ as well as where $x_n$ and $y_n$ are placed.}
    \label{fig:continuity}
\end{figure}
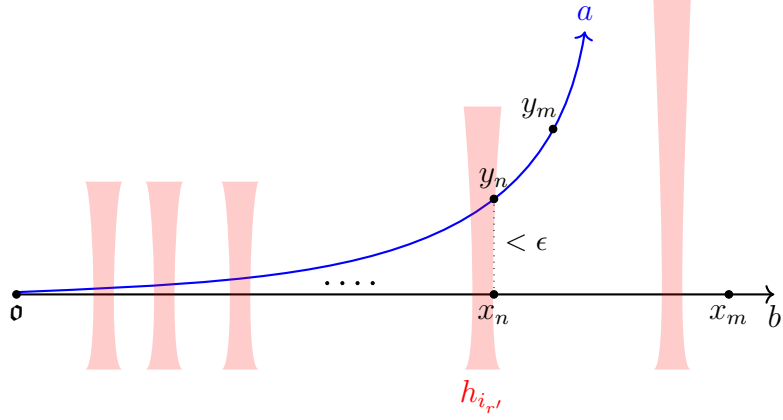

$$ \hat{d}\left(\ob, x_{m}\right) \geq \hat{d}\left(\ob, x_{n}\right)+\hat{d}\left(x_{n}, x_{m}\right)-C $$

 and 

$$
 \hat{d}\left(\ob, y_{m'} \right) \geq \hat{d}\left(\ob, y_{n}\right)+\hat{d}\left(y_{n}, y_{m'}\right)-C.
$$
Also, by the triangle inequality,
$$
\hat{d}\left(x_{m}, y_{m^{\prime}}\right) \leq \hat{d}\left(x_{m}, x_{n}\right)+\hat{d}\left(x_{n}, y_{n}\right)+\hat{d}\left(y_{n}, y_{m^{\prime}}\right).
$$

Thus, by the above inequalities, we get

$$
\begin{aligned}
\left(x_{m} \cdot y_{m^{\prime}}\right)_{\ob} &=\frac{1}{2}\left(\hat{d}\left(\ob, x_{m}\right)+\hat{d}\left(\ob, y_{m^{\prime}}\right)-\hat{d}\left(x_{m}, y_{m'}\right)\right)\\
& \geq \frac{1}{2}\left(\hat{d}\left(\ob, x_{n}\right)+\hat{d}\left(\ob, y_{n}\right)-\hat{d}\left(x_{n}, y_{n}\right)-2 C\right) \\
\end{aligned}
$$
Since $[\ob, x_n]$ and $[\ob, y_n]$ both cross the curtains $\{h_1, h_2, \cdots, h_{i_{r'}}\}$ we see that $\hat{d}\left(\ob, x_{n}\right) > r'$ and $\hat{d}\left(\ob, y_{n}\right) > r'$. Thus,
$$
\begin{aligned}
  \left(x_{m} \cdot y_{m^{\prime}}\right)_{\ob}  & \geq \frac{1}{2}\left(r^{\prime}+r^{\prime}-\epsilon-2 C\right) \\
&=r^{\prime}-\frac{1}{2} \epsilon- C.
\end{aligned}
$$
Since $r^{\prime}$ and $\epsilon$ were arbitrary, we can fix $\epsilon$ and choose $r^{\prime}$ such that $\left(x_{m} \cdot y_{m'}\right)_{\ob} > r$. Thus, $\varphi(a^\infty) \in U(s, r)$. This is true for any $a^\infty \in U_k(b^\infty)$. Hence, $\varphi$ is continuous. 

\end{proof}

\begin{remark}
Since the underlying sets of $X$ and $\widehat{X}$ are the same,  $\varphi$ is an $\text{Isom}\hspace{.1cm} X$-equivariant map.
\end{remark}

\begin{definition}[Sublinearly Morse boundary, Sublinearly Morse topology]
The $\kappa$-Morse boundary, $ \partial_{\kappa} X$, can be equipped with the \textit{sublinearly Morse topology}: Fix a base point $\mathfrak{o}$, let $ \xi \in \partial_{\kappa} X$, and let $b$ be the unique geodesic representative of $\xi$ that starts at $\ob$. For all $r>0$, we define $U_{\kappa}(b, r)$ to be the set of all points $\eta \in \partial_{\kappa} X$ such that for every $(K,C)-$quasi-geodesic $\beta$ representing $\eta$, starting at $\mathfrak{o}$, and satisfying $m_{b}(K,C) \leq \frac{r}{2 \kappa(r)}$, we have $\left.\beta\right|_{[0, r]} \subset \mathcal{N}_{\kappa}\left(a, m_{a}(\K,C)\right)$ (see \cite{QRT19}).

\end{definition}

The sublinear Morse topology as first defined in \cite{QRT19} was inspired by the Cashen-Mackay topology for the Morse boundary \cite{CM19}. This topology has some utility because it is metrizable and a quasi-isometry invariant unlike the cone topology. With the following lemma, we can import the result of Theorem \ref{Continuous} with respect to the sublinear Morse topology.

\begin{lemma}[Lemma 2.12 in \cite{IMZ21}]\label{cone topology coarser} Let $X$ be a proper CAT(0) space. The cone topology restricted to the set $\p_\K X$ is coarser than the sublinearly Morse topology.
\end{lemma} 

\begin{theorem}\label{sublinearly morse continuous}

    The map $\varphi: \partial_{\K}X \longrightarrow \p \widehat{X}$ is a well-defined, injective, and  continuous map when $\partial_{\K}X$ is endowed with the sublinear Morse topology.
\end{theorem}

Recent work of \cite{AM22} has shown new proving techniques when regarding cobounded projections to hyperbolic spaces. With Definition \ref{QuasiRuler} and Lemma \ref{A.11}, we can follow \cite{AM22}[Lemma 3.32] to upgrade $\varphi$ in Theorem \ref{Continuous} to a homeomorphism.

\begin{definition} \label{QuasiRuler} Let $X$ be a metric space, $C \geq 0$, and $I \subset \mathbb{R}$ a (possibly unbounded) closed interval. A path $\gamma: I \rightarrow X$ is an \textit{unparameterized C-quasi-ruler} if it satisfies the following conditions. \begin{itemize}
    \item $\forall t<s<r$, we have $d(\gamma(t), \gamma(s))+d(\gamma(s), \gamma(r)) \leq d(\gamma(t), \gamma(r))+C$.
    \item $\forall t_0 \in I$, we have $\limsup _{\left|t-t_0\right| \rightarrow 0} d\left(\gamma(t), \gamma\left(t_0\right)\right)<C$.
\end{itemize}

We see that, for any $\K$-contracting ray $b:[0,\infty) \longrightarrow X$, and any $t_0 \in [0, \infty)$, we have $\lim \sup_{|t-t_0|\to 0} \hat{d}(b(t), b(t_0)) \leq \sum \frac{1}{L^3}$. This, along with Lemma \ref{unparametized quasi geos} show that $b$ is a $C$-quasi-ruler for the constant $C$ in Lemma \ref{unparametized quasi geos}.

\begin{lemma}[Lemma A.11 in \cite{AM22}] \label{A.11}  Let $\widehat{X}$ be a $\delta$-hyperbolic space with basepoint $\ob \in X$, and let $\gamma, \gamma^{\prime}:[0, \infty) \rightarrow \widehat X$ be two $C$-quasi-rulers with unbounded image that start at $\ob$ and define points $[\{\gamma(n)\}], [\{\gamma'(m)\}]$ in $\p\widehat{X}$. If $x^{\prime} \in \gamma$ and $y^{\prime} \in \gamma^{\prime}$ are such that $d\left(\ob, x^{\prime}\right), d\left(\ob, y^{\prime}\right) \leq \liminf_{m, n \rightarrow \infty}\left(\gamma(n) \cdot \gamma'(m)\right)_{\ob}+C$, then
$$
\left(x^{\prime} \cdot y^{\prime}\right)_\ob \geq \min \left\{d\left(\ob, x^{\prime}\right), d\left(\ob, y^{\prime}\right)\right\}-C-2 \delta
$$
and
$$
d\left(x^{\prime}, y^{\prime}\right) \leq C+2 \delta+\left|d\left(\ob, x^{\prime}\right)-d\left(\ob, y^{\prime}\right)\right|
$$
\end{lemma}

\end{definition}

\begin{theorem}\label{homeo}
    When $\p_\K X$ is endowed with the subspace topology of the cone topology of $\p X$, the map $\varphi: \p_\K X \longrightarrow \p \widehat{X}$ is a homeomorphism onto its image.
\end{theorem}

\begin{proof}
    Continuity is already shown in Theorem \ref{Continuous}. We only need to show that, for any open set  $O \subset \p_\K X$, $\varphi(O)$ is open in $\varphi(\p_\K X)$ when endowed with the subspace topology.  Let $b^\infty \in O$. By the definition of being open in the curtain topology, there exists a curtain $k$ dual to $b$ such that $U_k(b^\infty) \subset O$. We have $b$ is $\K$-contracting, so there exists a $\K$-chain $\{h_i\}$ dual to $b$ with centers of poles $b(t_i)$ for each $h_i$. Since $k$ is dual to $b$, we have that there exists a $J$ such that $k$ separates $h_j$ from $\ob$ for all $j \geq J$.

    Let $r >0$. Choose $t_r$ such that for all $y \in X$ on the same side of $k$ as $\ob$, $\hat{d}(y, b(t_r)) > r$. Note, this is possible since, for large enough $t$, $b(t)$ will be on the side of $h_J$ opposite of $\ob$. Since $b([J, \infty))$ is unbounded, such a $b(t_r)$ will exist. Put $R = \hat{d}(\ob,b(t_r))$. Set $V = U(\varphi(b^\infty),R) \cap \varphi(\p_\K X)$. We claim $\varphi(b^\infty) \in V \subseteq \varphi(U_k(b^\infty)) \subseteq \varphi(O)$. 

    Indeed, let $\eta \in V$, and put $a$ as the geodesic emanating from $\ob$ such that $\varphi(a^\infty) = \eta$. So $[\{a(n)\}] = \eta$. Let $m\in \mathbb{N}$ be the smallest $m$ such that $\hat{d}(\ob, a(m)) \geq R$ for $a(m) \in \{a(n)\}$. Since $a$ is a $C$-quasi-ruler, we have that there exists a $C$ such that $\hat{d}(\ob, a(m)) \leq R + C$. By Lemma \ref{A.11}, we get that $$\hat{d}(b(t_r), a(m)) \leq C + 2\delta +|\hat{d}(\ob, b(t_r))-\hat{d}(\ob, a(m))|\leq 2C + 2\delta,$$ where $\delta$ is the hyperbolicity constant of $\widehat{X}$. Since $r$ was arbitrary, we can choose $r$ large enough to force $a(m)$ to be on the same side of $k$ as $b(t_r)$. This means $a$ crosses $k$. In other words, $a^\infty \in U_k(b)$. This completes the proof.
\end{proof}

  \section{ A Persistent Shadow Characterization}\label{Section 5}
We now give our second characterization of $\K$-Morse rays (Theorem \hyperlink{TheoremD}{D} in the introduction). Just as how $\K$-Morse rays in mapping class groups project to curve graphs in a sublinearly scaled way, we find an equivalent notion in the CAT(0) setting.

  \begin{definition}[$\K$-persistent shadow, persistent shadow constant]
       A geodesic ray $b$ with infinite diameter has a \textit{$\K$-persistent shadow} in $\widehat X$ if there exists a $C>0$ such that for all $s<t$, \begin{center}
    $\hat d(b(s),b(t)) \geq C \cdot \displaystyle \frac{t-s}{\K(t)} - C.$
\end{center}
We refer to $C$ above as the \textit{persistent shadow constant.}
  \end{definition}

We show how this characterization is connected to $\kappa$-contracting in the next Theorem \ref{persistent characterization}. The following lemma works in dualizing $L$-chains that meet a geodesic, and is used in the proof of Theorem \ref{persistent characterization}.

\begin{lemma}[Lemma 2.21 in \cite{PSZ22}] \label{dualizing}
     Let $L, n \in \mathbb{N}$, let $\left\{h_1, \ldots, h_{(4 L+10) n}\right\}$ be an $L$-chain, and suppose that $A, B \subset X$ are separated by every $h_i$. For any $x \in A$ and $y \in B$, the sets $A$ and $B$ are separated by an $L$-chain of length at least $n+1$ all of whose elements are dual to $[x, y]$ and separate $h_1$ from $h_{(4 L+10) n}$.
\end{lemma}

 \begin{theorem}
     
 \label{persistent characterization}
 Let $b$ be geodesic ray in a CAT(0) space $X$ emanating from $\ob$ with infinite diameter projection onto $\widehat{X}$.\begin{itemize}
     \item If $b$ is $\K$-contracting and $\K^4$ is sublinear, then $b$ has a $\K^4$-persistent shadow in the $\hd$ metric.
     \item If $b$ has a $\K$-persistent shadow in the $\hat{d}$ metric and $\K^2$ is sublinear, then $b$ is $\K^2$-contracting.
 \end{itemize}

 \end{theorem}
 \begin{proof}
For the forward direction, put $D>0$ as the excursion constant, and denote $\{b(t_i)\}$ as the center of poles of the dual curtains in the $\K$-chain $\{h_i\}$. We follow the same process as in Lemma \ref{unbnd}. That is, for any $s\leq t$, we have that there exists a maximal $j$ and minimal $i$ such that $t_j \leq s\leq t \leq t_i$ (if $s \leq t_1$, we choose $t_0 = 0$). Following the proof of Lemma \ref{unbnd} gives us $$i-j\geq \frac{t_i-t_j}{D\K(t_i)}.$$ Thus, since $\lceil D\K(t_i) \rceil = m$ for some $m \in \mathbb{N},$ we get that $d_m(b(s),b(t)) \geq i-j-2$ and \begin{align*}
    \frac{d_m(b(s),b(t))}{m^3} &\geq \frac{\frac{t_i-t_j}{D\K(t_i)} - 2}{m^3}\\
                                 &\geq \frac{t_i-t_j}{D\K(t_i)(D\K(t_i) +1)^3} - 2\\
                                 &\geq \frac{t-s}{D'\K(t)(D'\K(t) +1)^3} - 2.\\
\end{align*}

for some $D'>0$ dependent on $D$ by Lemma \ref{3.2}. Thus, $b$ has a  $\K^4$-persistent shadow.

For the reverse direction, put $C$ as the persistent shadow constant. Since $b$ has infinite diameter and $\K^2$ is sublinear, consider a sequence $\{t_i\}$ such that $t_{i+1} - t_i = \frac{D}{C^2}\K^2(t_{i+1})$ for some finite $D>0$. Note that, for any $i$, $\lceil \K(t_{i+1})\rceil \geq 2$, and $\sum_{L=2}^{\infty}\frac{1}{L^3} \leq \sum_{L=2}^{\infty}\frac{1}{L^2} \leq \frac{2}{3}$. Thus, 
     \begin{align*}
         \frac{D}{C}\K(t_{i+1}) - C &=  C \cdot \frac{t_{i+1}-t_i}{\K(t_{i+1})} - C\\
                                    &\leq \hat{d}(b(t_i), b(t_{i+1})) \\
                                    &= \displaystyle \sum_{L=1}^{\lceil \K(t_{i+1)}\rceil-1}\frac{d_L(b(t_i),b(t_{i+1}))}{L^3} +  \displaystyle \sum_{L=\lceil \K(t_{i+1})\rceil}^{\infty}\frac{d_L(b(t_i),b(t_{i+1}))}{L^3}\\
                                    &\leq \displaystyle \sum_{L=1}^{\lceil \K(t_{i+1})\rceil-1}\frac{d_L(b(t_i),b(t_{i+1}))}{L^3} +  \displaystyle \sum_{L=\lceil \K(t_{i+1})\rceil}^{\infty}\frac{d(b(t_i),b(t_{i+1})) + 1}{L^3}\\
                                    &= \displaystyle \sum_{L=1}^{\lceil \K(t_{i+1})\rceil-1}\frac{d_L(b(t_i),b(t_{i+1}))}{L^3} +  \displaystyle \sum_{L=\lceil \K(t_{i+1})\rceil}^{\infty}\frac{\frac{D}{C}\K^2(t_{i+1}) + 1}{L^3}\\
                                    &\leq \displaystyle \sum_{L=1}^{\lceil \K(t_{i+1})\rceil-1}\frac{d_L(b(t_i),b(t_{i+1}))}{L^3} +  \displaystyle \frac{D}{C}\K(t_{i+1})\cdot\left(\sum_{L=\lceil \K(t_{i+1})\rceil}^{\infty}\frac{1}{L^2}\right)  + 1\\
                                     &\leq \displaystyle \sum_{L=1}^{\lceil \K(t_{i+1})\rceil-1}\frac{d_L(b(t_i),b(t_{i+1}))}{L^3} +  \displaystyle \frac{D}{C}\K(t_{i+1})\cdot \frac{2}{3}  + 1.\\
     \end{align*}
     That is, \begin{align}
        \displaystyle \frac{1}{3C}D\K(t_{i+1}) - C -1 \leq \displaystyle \sum_{L=1}^{\lceil \K(t_{i+1})\rceil-1}\frac{d_L(b(t_i),b(t_{i+1}))}{L^3}. \label{inequality}
     \end{align}

     Since $D$ was arbitrary, we now choose $D=3003C$ so that the left hand side of (\ref{inequality}) is  greater than $1000\K(t_{i+1}).$ We claim that for some $L$ with $1 \leq L \leq \lceil \K(t_{i+1})\rceil-1$, we get $d_L(b(t_i), b(t_{i+1})) \geq 8\K(t_{i+1})+21$. Indeed, if not, then 
     
     \begin{align*}
        \displaystyle \sum_{L=1}^{\lceil \K(t_{i+1})\rceil-1}\frac{d_L(b(t_i),b(t_{i+1}))}{L^3} &\leq \displaystyle \sum_{L=1}^{\lceil \K(t_{i+1})\rceil-1}\frac{8\K(t_{i+1})+ 21}{L^3} \leq 2(8\K(t_{i+1})+ 21) \\
     \end{align*}
which contradicts (\ref{inequality}). By Lemma \ref{dualizing}, we get that there are three $L$-separated curtains dual to $[b(t_i), b(t_{i+1})]$. These curtains, by construction, are also $\K(t_{i+1})$ separated. This is true for all $i$. We choose $h_i$ to be the dual curtain of $[b(t_i), b(t_{i+1})]$ that is closest to $b(t_i)$, and put $s_i$ to be the center of the pole of $h_i$. Then, we see that each of the pairs $\{h_i, h_{i+1}\}$ are $\K(t_{i+1})$-separated since there are two $\K(t_{i+1})$-separated curtains that separate $h_i$ and $h_{i+1}$. See Figure \ref{fig:K chain}.
\begin{figure}[ht]
    \centering
    \begin{tikzpicture}[scale=1.3]    
    \draw [blue,thick, ->](1,-2) .. controls (5,-2) and (8,-2) .. (11,-2);
    \node [below, blue] at (11,-2) {$b$};

    % \draw[red!50, fill = red!50, opacity=0.3  ] plot((3-1,-3) -- (3.1-1,-3) -- (3.1-1,-.5) -- (3-1,-.5) -- cycle;

    % \draw[red!50, fill = red!50, opacity=0.3  ] plot((2+1.6-1,-3) -- (2.1+1.6-1,-3) -- (2.1+1.6-1,-.5) -- (2+1.6-1,-.5) -- cycle;
    % \draw[red!50, fill = red!50, opacity=0.3  ] plot((2+2.2-1,-3) -- (2.1+2.2-1,-3) -- (2.1+2.2-1,-.5) -- (2+2.2-1,-.5) -- cycle;

    % \draw[red!50, fill = red!50, opacity=0.3  ] plot((3+3-1,-3) -- (3.1+3-1,-3) -- (3.1+3-1,-.5) -- (3+3-1,-.5) -- cycle;

    % \draw[red!50, fill = red!50, opacity=0.3  ] plot((2+1.6+3-1,-3) -- (2.1+1.6+3-1,-3) -- (2.1+1.6+3-1,-.5) -- (2+1.6+3-1,-.5) -- cycle;
    % \draw[red!50, fill = red!50, opacity=0.3  ] plot((2+2.2+3-1,-3) -- (2.1+2.2+3-1,-3) -- (2.1+2.2+3-1,-.5) -- (2+2.2+3-1,-.5) -- cycle;

    % \draw[red!50, fill = red!50, opacity=0.3  ] plot((3+6-1,-3) -- (3.1+6-1,-3) -- (3.1+6-1,-.5) -- (3+6-1,-.5) -- cycle;

    % \draw[red!50, fill = red!50, opacity=0.3  ] plot((2+1.6+6-1,-3) -- (2.1+1.6+6-1,-3) -- (2.1+1.6+6-1,-.5) -- (2+1.6+6-1,-.5) -- cycle;
    % \draw[red!50, fill = red!50, opacity=0.3  ] plot((2+2.2+6-1,-3) -- (2.1+2.2+6-1,-3) -- (2.1+2.2+6-1,-.5) -- (2+2.2+6-1,-.5) -- cycle;

\fill[red, opacity=0.2] (1.95,-3) .. controls(2,-3)  and (2,-.5)   .. (1.95,-.5) -- (2.15,-.5).. controls (2.1,-.5)  and (2.1,-3) .. (2.15,-3) -- cycle;

\fill[red, opacity=0.2] (1.95+.6,-3) .. controls(2+.6,-3)  and (2+.6,-.5)   .. (1.95+.6,-.5) -- (2.15+.6,-.5).. controls (2.1+.6,-.5)  and (2.1+.6,-3) .. (2.15+.6,-3) -- cycle;

\fill[red, opacity=0.2] (1.95+1.2,-3) .. controls(2+1.2,-3)  and (2+1.2,-.5)   .. (1.95+1.2,-.5) -- (2.15+1.2,-.5).. controls (2.1+1.2,-.5)  and (2.1+1.2,-3) .. (2.15+1.2,-3) -- cycle;

\fill[red, opacity=0.2] (1.95+3,-3) .. controls(2+3,-3)  and (2+3,-.5)   .. (1.95+3,-.5) -- (2.15+3,-.5).. controls (2.1+3,-.5)  and (2.1+3,-3) .. (2.15+3,-3) -- cycle;

\fill[red, opacity=0.2] (1.95+3.6,-3) .. controls(2+3.6,-3)  and (2+3.6,-.5)   .. (1.95+3.6,-.5) -- (2.15+3.6,-.5).. controls (2.1+3.6,-.5)  and (2.1+3.6,-3) .. (2.15+3.6,-3) -- cycle;

\fill[red, opacity=0.2] (1.95+4.2,-3) .. controls(2+4.2,-3)  and (2+4.2,-.5)   .. (1.95+4.2,-.5) -- (2.15+4.2,-.5).. controls (2.1+4.2,-.5)  and (2.1+4.2,-3) .. (2.15+4.2,-3) -- cycle;

\fill[red, opacity=0.2] (1.95+6,-3) .. controls(2+6,-3)  and (2+6,-.5)   .. (1.95+6,-.5) -- (2.15+6,-.5).. controls (2.1+6,-.5)  and (2.1+6,-3) .. (2.15+6,-3) -- cycle;

\fill[red, opacity=0.2] (1.95+6.6,-3) .. controls(2+6.6,-3)  and (2+6.6,-.5)   .. (1.95+6.6,-.5) -- (2.15+6.6,-.5).. controls (2.1+6.6,-.5)  and (2.1+6.6,-3) .. (2.15+6.6,-3) -- cycle;

\fill[red, opacity=0.2] (1.95+7.2,-3) .. controls(2+7.2,-3)  and (2+7.2,-.5)   .. (1.95+7.2,-.5) -- (2.15+7.2,-.5).. controls (2.1+7.2,-.5)  and (2.1+7.2,-3) .. (2.15+7.2,-3) -- cycle;

    \node[below, red] at (2.05, -3) {$h_i$};

    \node[below, red] at (5.05, -3) {$h_{i+1}$};

    \node[below, red] at (8.05, -3) {$h_{i+2}$};

    \draw[fill] (1.5, -2) circle [radius=0.05];
    \node[below] at (1.5, -2) {$b(t_i)$};

    \draw[fill] (4, -2) circle [radius=0.05];
    \node[below] at (4, -2) {$b(t_{i+1})$};

    \draw[fill] (7.3, -2) circle [radius=0.05];
    \node[below] at (7.3, -2) {$b(t_{i+2})$};

    \draw[fill] (10.3, -2) circle [radius=0.05];
    \node[below] at (10.3, -2) {$b(t_{i+3})$};

    \draw [opacity=0.6, decorate,
    decoration = {brace}] (3.35,-3.07) -- (2.55,-3.07) ;

    \draw[opacity=0.6] (2.95,-3.1) to (2.95,-3.45);
    \node [opacity=0.6, below] at (2.95,-3.4) {$\kappa(t_{i+1})$-separated};

  \draw [opacity=0.6, decorate,
    decoration = {brace}] (3.35+3,-3.07) -- (2.55+3,-3.07) ;

    \draw[opacity=0.6] (2.95+3,-3.1) to (2.95+3,-3.45);
    \node [below,opacity=0.6] at (2.95+3,-3.4) {$\kappa(t_{i+2})$-separated};

      \draw [opacity=0.6,decorate,
    decoration = {brace}] (3.35+6,-3.07) -- (2.55+6,-3.07) ;

    \draw [opacity=0.6](2.95+6,-3.1) to (2.95+6,-3.45);
    \node [below, opacity=0.6] at (2.95+6,-3.4) {$\kappa(t_{i+3})$-separated};

    \end{tikzpicture}
    \caption{Creating a dual $\K$-chain to $b$. Since each pair $\{h_i, h_{i+1}\}$ will have two $\K(t_{i+1})$-separated curtains that separate the pair, we get that $h_i$ and $h_{i+1}$ will be $\K(t_{i+1})$-separated.}
    \label{fig:K chain}
\end{figure}
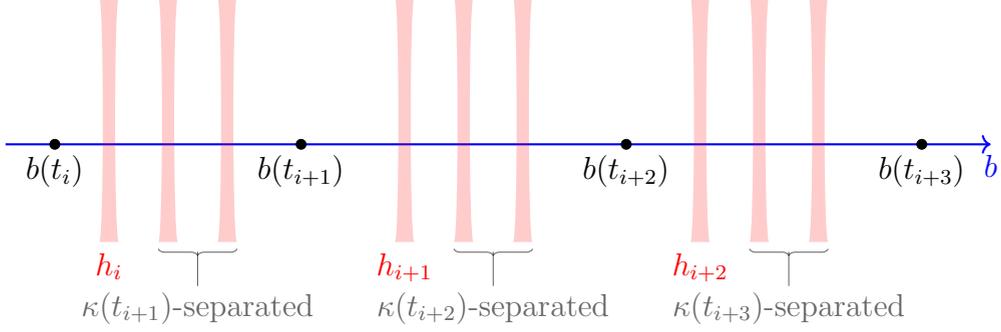
 Also,

\begin{align*}
    s_{i+1} - s_i &\leq t_{i+2} - t_i\\
                  &\leq 2\frac{D}{C^2}\K^2(t_{i+2})\\
                  &\leq D'\K^2(t_{i+1})\\
                  &\leq D'\K^2(t_{s_{i+1}})\\
\end{align*}
for some $D' > 0$ dependent on $C$ by Lemma \ref{3.2}. Thus, we have an infinite chain of dual curtains $\{h_i\}$ such that $h_i$ and $h_{i+1}$ are $\K(t_{i+1})$-separated (which implies $\K^2(t_{i+1})$-separated) and $t_{i+1} - t_i = \frac{D}{C^2}\K^2(t_{i+1})$  for all $i$. We conclude that $b$ is $\K^2$-contracting.
 \end{proof}

\color{black}

\section{ Hyperbolicity Criterion Using Curtain Grids}\label{Section 6}

Work of Genevois in \cite{Gen16} has shown a hyperbolicity criterion for CAT(0) cube complexes. We work to generalize this criterion to the CAT(0) setting.

\begin{definition}[Curtain Grid, $E$-thin]
    Two chains of curtains $\mathcal{H} = \{h_1, \cdots, h_n\}$ and $\mathcal{K}= \{k_1, \cdots, k_m\}$ form a \textit{curtain grid} if every curtain of $\mathcal{H}$ crosses every curtain of $\mathcal{K}$. We denote a curtain grid as $(\mathcal{H}, \mathcal{K})$. Given $E >0$, a curtain grid is said to be \textit{$E$-thin} if min$\{|\mathcal{H}|,|\mathcal{K}|\} \leq E$.
\end{definition}
Comparing to the cube complex setting, two chains of hyperplanes forming a grid will equate to a region of ``flatness" with the intuition of larger grids equating to larger areas of ``flatness". Thus, if one has an upper bound on how large these hyperplane grids can get, one could expect a notion of hyperbolicity. This is precisely what Genevois proves in \cite{Gen16}, and the following theorem uses curtains to get a similar criterion for the CAT(0) setting.
\begin{theorem}\label{grid theorem}
    Let $X$ be a CAT(0) space. Then $X$ is hyperbolic if and only if every curtain grid is $E$-thin for some uniform $E > 0$.
\end{theorem}

\begin{proof}
    If $X$ is a hyperbolic space, then all geodesics in $X$ are uniformly $D$-contracting for some constant $D$. By Theorem \hyperlink{TheoremC}{C} for $\kappa \equiv 1$ (this is equivalent to Theorem 4.2 in \cite{PSZ22}), we have there exists an $L=L(D)$ such for any two curtains $l_1,l_2$ dual to the same geodesic that are also of distance at least $L$ apart, we get that $l_1$ and $l_2$ are $L$-separated.
    
    Consider any curtain grid $(\mathcal{H}, \mathcal{K})$. So $\mathcal{H} = \{h_1, \cdots, h_n\}$ for some $n$. Suppose $n > 5L + 8$. Let $x \in h_1^-$ and $y \in h_n^+$ and denote the unique geodesic between $x$ and $y$ as $[x,y]$. Each $h_i$ crosses $[x,y]$ by nature of $\mathcal{H}$ being a chain. Since all of the curtains in $\mathcal{H}$ are disjoint, there exists $a_i \in [x,y]\cap h_i^+\cap h_{i+1}^-$ for all $i$. Denote the curtains dual to $[x,y]$ and centered at $a_i$ by $l_i$. We now consider the chain  $\mathcal{L} = \{l_{L+3}, l_{2L+4}, l_{n-(2L+3)}, l_{n-(L+2)}\}$. Each curtain in $\mathcal{L}$ is distance at least $L$ apart from the next curtain in $\mathcal{L}$. Thus, $\mathcal{L}$ is an $L$-chain. Notice the subchain of $\mathcal{H}$ that is $\{h_2, \cdots, h_{L+2}\}$ intersects nontrivially with $l_{L+3}^-$. Thus, $h_1$ cannot intersect $l_{2L+4}$ or else we would contradict $L$-separation between $l_{L+3}$ and $ l_{2L+4}$.  Similarly, $h_n$ cannot intersect $l_{n-(2L+3)}$. See Figure \ref{grid fig}. Thus, $l_{2L+4}$ and $l_{n- (2L+3)}$ are $L$-separated curtains that both separate $h_1$ and $h_n$. Since all curtains in $\mathcal{K}$ cross both $h_1$ and $h_n$, they must also cross both $l_{2L+4}$ and $l_{n- (2L+4)}$. The $L$-separability of $l_{2L+4}$ and $l_{n- (2L+3)}$ implies that $|\mathcal{K}| \leq L \leq 5L+8$. Hence, all curtain grids are uniformly $(5L+8)$-thin.

             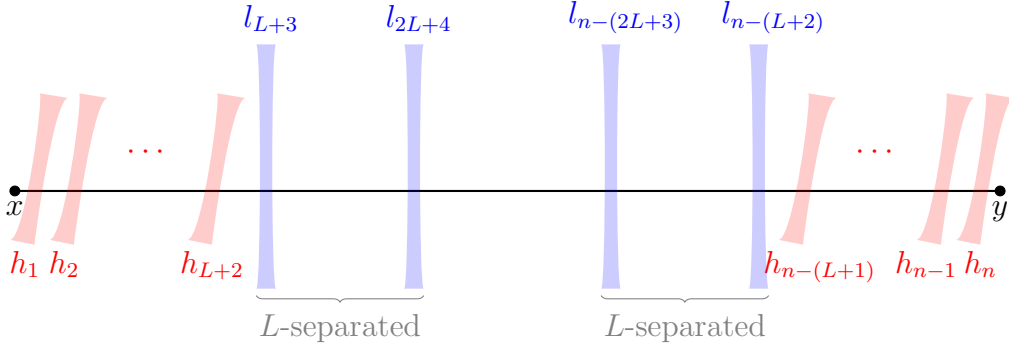
\begin{figure}[ht]
    \centering
    \begin{tikzpicture}[scale=1.3]    
    \draw [thick](1,-2) .. controls (5,-2) and (8,-2) .. (11,-2);
    \node [below] at (11,-2) {$y$};
    \node [below] at (1,-2) {$x$};
    \draw[fill] (11, -2) circle [radius=0.05];
    \draw[fill] (1, -2) circle [radius=0.05];

% \fill[red, opacity=0.2] (1.95,-2.5) .. controls(2,-2.5)  and (2,-1)   .. (1.95,-1) -- (2.15,-1).. controls (2.1,-1)  and (2.1,-2.5) .. (2.15,-2.5) -- cycle;

\fill[red, opacity=0.2] (1.95-.2-.8,-2.5) .. controls(2-.1-.8,-2.5)  and (2+.1-.8,-1)   .. (1.95+.1-.8,-1) -- (2.15+.2-.8,-1-.05).. controls (2.1+.1-.8,-1)  and (2.1-.1-.8,-2.5) .. (2.15-.15-.8,-2.5-.05) -- cycle;

\fill[red, opacity=0.2] (1.95-.2-.8+.4,-2.5) .. controls(2-.1-.8+.4,-2.5)  and (2+.1-.8+.4,-1)   .. (1.95+.1-.8+.4,-1) -- (2.15+.2-.8+.4,-1-.05).. controls (2.1+.1-.8+.4,-1)  and (2.1-.1-.8+.4,-2.5) .. (2.15-.15-.8+.4,-2.5-.05) -- cycle;

\fill[red, opacity=0.2] (1.95-.2+1,-2.5) .. controls(2-.1+1,-2.5)  and (2+.1+1,-1)   .. (1.95+.1+1,-1) -- (2.15+.2+1,-1-.05).. controls (2.1+.1+1,-1)  and (2.1-.1+1,-2.5) .. (2.15-.15+1,-2.5-.05) -- cycle;

   \node[below, red] at (1.1, -2.5) {$h_1$};
   \node[below, red] at (1.5, -2.5) {$h_2$};
   \node[below, red] at (3, -2.5) {$h_{L+2}$};

   \node[red] at (2.35,-1.6) {$\cdots$};

\fill[red, opacity=0.2] (1.95-.2+8.8,-2.5) .. controls(2-.1+8.8,-2.5)  and (2+.1+8.8,-1)   .. (1.95+.1+8.8,-1) -- (2.15+.2+8.8,-1-.05).. controls (2.1+.1+8.8,-1)  and (2.1-.1+8.8,-2.5) .. (2.15-.15+8.8,-2.5-.05) -- cycle;

\fill[red, opacity=0.2] (1.95-.2+8.4,-2.5) .. controls(2-.1+8.4,-2.5)  and (2+.1+8.4,-1)   .. (1.95+.1+8.4,-1) -- (2.15+.2+8.4,-1-.05).. controls (2.1+.1+8.4,-1)  and (2.1-.1+8.4,-2.5) .. (2.15-.15+8.4,-2.5-.05) -- cycle;

\fill[red, opacity=0.2] (1.95-.2+7,-2.5) .. controls(2-.1+7,-2.5)  and (2+.1+7,-1)   .. (1.95+.1+7,-1) -- (2.15+.2+7,-1-.05).. controls (2.1+.1+7,-1)  and (2.1-.1+7,-2.5) .. (2.15-.15+7,-2.5-.05) -- cycle;

 \node[below, red] at (9.17, -2.5) {$h_{n-(L+1)}$};
 \node[below, red] at (10.25, -2.5) {$h_{n-1}$};
 \node[below, red] at (10.8, -2.5) {$h_{n}$};

 \node[red] at (9.75,-1.6) {$\cdots$};

\fill[blue, opacity=0.2] (1.95+1.5,-3) .. controls(2+1.5,-3)  and (2+1.5,-.5)   .. (1.95+1.5,-.5) -- (2.15+1.5,-.5).. controls (2.1+1.5,-.5)  and (2.1+1.5,-3) .. (2.15+1.5,-3) -- cycle;

\fill[blue, opacity=0.2] (1.95+1+1.5+.5,-3) .. controls(2+1+1.5+.5,-3)  and (2+1+1.5+.5,-.5)   .. (1.95+1+1.5+.5,-.5) -- (2.15+1+1.5+.5,-.5).. controls (2.1+1+1.5+.5,-.5)  and (2.1+1+1.5+.5,-3) .. (2.15+1+1.5+.5,-3) -- cycle;

\fill[blue, opacity=0.2] (1.95+1+4,-3) .. controls(2+1+4,-3)  and (2+1+4,-.5)   .. (1.95+1+4,-.5) -- (2.15+1+4,-.5).. controls (2.1+1+4,-.5)  and (2.1+1+4,-3) .. (2.15+1+4,-3) -- cycle;

\fill[blue, opacity=0.2] (1.95+1+5.5,-3) .. controls(2+1+5.5,-3)  and (2+1+5.5,-.5)   .. (1.95+1+5.5,-.5) -- (2.15+1+5.5,-.5).. controls (2.1+1+5.5,-.5)  and (2.1+1+5.5,-3) .. (2.15+1+5.5,-3) -- cycle;

    % \node[below, blue] at (3.5, -3) {$l_{L+3}$};
    \node[above, blue] at (3.6, -.5) {$l_{L+3}$};

    % \node[below, blue] at (3.5+1.5, -3) {$l_{L+3}$};
    \node[above, blue] at (3.6+1.5, -.5) {$l_{2L+4}$};

    % \node[below, blue] at (3.5+3.5, -3) {$l_{n-(2L+3)}$};
    \node[above, blue] at (3.7+3.5, -.5) {$l_{n-(2L+3)}$};

    % \node[below, blue] at (3.5+5, -3) {$l_{L+3}$};
    \node[above, blue] at (3.7+5, -.5) {$l_{n-(L+2)}$};

    \draw [gray, decorate,
    decoration = {brace}] (5.15,-3.07) -- (3.45,-3.07) ;

    \node [gray, below] at (4.3,-3.15) {$L$-separated};

    \draw [gray, decorate,
    decoration = {brace}] (5.15+3.5,-3.07) -- (3.45+3.5,-3.07) ;

    \node [gray, below] at (4.3+3.5,-3.15) {$L$-separated};

    \end{tikzpicture}
    \caption{Since $l_{L+3}$ and $l_{2L+4}$ are $L$-separated and $\{h_2, \cdots , h_{L+2}\}$ is a chain of length $L$, we must have that $h_1$ cannot meet $l_{2L+4}$. Similarly, $h_n$ cannot meet $l_{n-(2L+3)}$. }
    \label{grid fig}
\end{figure}

    For the reverse direction, let all curtain grids be $E$-thin for some uniform $E$. This means that any two curtains $h_1, h_j$ dual to the same geodesic that are also greater than $E$ distance away from each other must be $E$-separated. Indeed, such a situation would give a chain of $E+1$ curtains $\{h_1, h_2, \cdots, h_j\}$ all dual to the same geodesic, so any grid made with this chain and some other chain $\mathcal{K}$ must give $|\mathcal{K}| \leq E$. Consider any $x,y \in X$. Then, there exists an $n \in \mathbb{Z}_{\geq0}$ such that $$n(E+2) \leq d(x,y) \leq (n+1)(E+2).$$ Thus, there exist a chain $c = \{h_1, \cdots, h_{n(E+2)-1}\}$ such that each $h_i$ is dual to $[x,y]$. We can then conclude that $\{h_1, h_{(E+2)}, h_{2(E+2)}, \cdots, h_{(n-1)(E+2)}\}$ is an $E$-chain of length $n$. So $d_E(x,y) \geq n$. Hence, we get $$d_E(x,y) \leq d(x,y) \leq d_E(x,y)(E+2) +(E+2).$$ This gives that $X$ is quasi-isometric to $X_E$, a hyperbolic space. we conclude that $X$ is hyperbolic.
\end{proof}

\bibliography{Characterization.bib}{}

\begin{thebibliography}{CFFT22}

\bibitem[Ago13]{Ago13}
Ian Agol.
\newblock The virtual haken conjecture (with an appendix by ian agol, daniel
  groves and jason manning).
\newblock {\em Documenta Mathematica}, 18:1045--1087, 2013.

\bibitem[AIM22]{AM22}
Carolyn Abbott and Merlin Incerti-Medici.
\newblock Hyperbolic projections and topological invariance of sublinearly
  morse boundaries.
\newblock {\em arXiv preprint arXiv:2212.09539}, 2022.

\bibitem[Bar22a]{LeB1}
Corentin~Le Bars.
\newblock Central limit theorem on cat (0) spaces with contracting isometries.
\newblock {\em arXiv preprint arXiv:2209.11648}, 2022.

\bibitem[Bar22b]{LeB2}
Corentin~Le Bars.
\newblock Random walks and rank one isometries on cat (0) spaces.
\newblock {\em arXiv preprint arXiv:2205.07594}, 2022.

\bibitem[Beh06]{Ber06}
Jason~A Behrstock.
\newblock Asymptotic geometry of the mapping class group and teichm{\"u}ller
  space.
\newblock {\em Geometry \& Topology}, 10(3):1523--1578, 2006.

\bibitem[BH99]{BH99}
Martin~R. Bridson and Andr\'{e} Haefliger.
\newblock {\em Metric spaces of non-positive curvature}, volume 319 of {\em
  Grundlehren der mathematischen Wissenschaften [Fundamental Principles of
  Mathematical Sciences]}.
\newblock Springer-Verlag, Berlin, 1999.

\bibitem[BHS17]{BHS17}
Jason Behrstock, Mark Hagen, and Alessandro Sisto.
\newblock Hierarchically hyperbolic spaces, i: Curve complexes for cubical
  groups.
\newblock {\em Geometry \& Topology}, 21(3):1731--1804, 2017.

\bibitem[BHS21]{BHS21}
Jason Behrstock, Mark~F Hagen, and Alessandro Sisto.
\newblock Quasiflats in hierarchically hyperbolic spaces.
\newblock {\em Duke Mathematical Journal}, 170(5):909--996, 2021.

\bibitem[BK02]{BK02}
Nadia Benakli and Ilya Kapovich.
\newblock Boundaries of hyperbolic groups.
\newblock {\em arXiv preprint math/0202286}, 2002.

\bibitem[CD95]{CD95}
Ruth Charney and Michael~W Davis.
\newblock Finite k (; 1) s for artin groups.
\newblock {\em Ann. of Math. Stud., Princeton Univ. Press}, 138:110--124, 1995.

\bibitem[CFFT22]{CFFT22}
Kunal Chawla, Behrang Forghani, Joshua Frisch, and Giulio Tiozzo.
\newblock The poisson boundary of hyperbolic groups without moment conditions.
\newblock {\em arXiv preprint arXiv:2209.02114}, 2022.

\bibitem[Cho22]{choi22}
Inhyeok Choi.
\newblock Random walks and contracting elements i: Deviation inequality and
  limit laws.
\newblock {\em arXiv preprint arXiv:2207.06597v1}, 2022.

\bibitem[Cho23]{Choi}
Inhyeok Choi.
\newblock Random walks and contracting elements iv: Sublinearly morse boundary.
\newblock {\em to appear}, 2023.

\bibitem[CK00]{CK00}
Christopher~B. Croke and Bruce Kleiner.
\newblock Spaces with nonpositive curvature and their ideal boundaries.
\newblock {\em Topology}, 39(3):549--556, 2000.

\bibitem[CM19]{CM19}
Christopher~H. Cashen and John~M. Mackay.
\newblock A metrizable topology on the contracting boundary of a group.
\newblock {\em Trans. Amer. Math. Soc.}, 372(3):1555--1600, 2019.

\bibitem[Cor17]{Cor16}
Matthew Cordes.
\newblock Morse boundaries of proper geodesic metric spaces.
\newblock {\em Groups Geom. Dyn.}, 11(4):1281--1306, 2017.

\bibitem[Cor19]{Cor19}
Matthew Cordes.
\newblock A survey on {M}orse boundaries and stability.
\newblock In {\em Beyond hyperbolicity}, volume 454 of {\em London Math. Soc.
  Lecture Note Ser.}, pages 83--116. Cambridge Univ. Press, Cambridge, 2019.

\bibitem[CS15]{CS15}
Ruth Charney and Harold Sultan.
\newblock Contracting boundaries of {$\rm CAT(0)$} spaces.
\newblock {\em J. Topol.}, 8(1):93--117, 2015.

\bibitem[DMS20]{DMS20}
Matthew~G Durham, Yair~N Minsky, and Alessandro Sisto.
\newblock Stable cubulations, bicombings, and barycenters.
\newblock {\em arXiv preprint arXiv:2009.13647}, 2020.

\bibitem[DR09]{DR10}
Moon Duchin and Kasra Rafi.
\newblock Divergence of geodesics in teichm{\"u}ller space and the mapping
  class group.
\newblock {\em Geometric and Functional Analysis}, 19:722--742, 2009.

\bibitem[DZ22]{DZ22}
Matthew~Gentry Durham and Abdul Zalloum.
\newblock The geometry of genericity in mapping class groups and teichmuller
  spaces via cat (0) cube complexes.
\newblock {\em arXiv preprint arXiv:2207.06516}, 2022.

\bibitem[Gen16]{Gen16}
Anthony Genevois.
\newblock Coning-off cat (0) cube complexes.
\newblock {\em arXiv preprint arXiv:1603.06513}, 2016.

\bibitem[Gen19]{Gen20b}
Anthony Genevois.
\newblock Hyperbolicities in {${\rm CAT}(0)$} cube complexes.
\newblock {\em Enseign. Math.}, 65(1-2):33--100, 2019.

\bibitem[GQR22]{GQR22}
Ilya Gekhtman, Yulan Qing, and Kasra Rafi.
\newblock Genericity of sublinearly morse directions in cat (0) spaces and the
  teichm$\backslash$" uller space.
\newblock {\em arXiv preprint arXiv:2208.04778}, 2022.

\bibitem[Gro87]{Gro87}
M.~Gromov.
\newblock Hyperbolic groups.
\newblock In {\em Essays in group theory}, volume~8 of {\em Math. Sci. Res.
  Inst. Publ.}, pages 75--263. Springer, New York, 1987.

\bibitem[He23]{He23}
Vivian He.
\newblock Equivalent topologies on the contracting boundary.
\newblock {\em Glasnik matemati{\v{c}}ki}, 58(1):75--83, 2023.

\bibitem[HHP20]{HHP22}
Thomas Haettel, Nima Hoda, and Harry Petyt.
\newblock Coarse injectivity, hierarchical hyperbolicity, and
  semihyperbolicity.
\newblock {\em arXiv preprint arXiv:2009.14053}, 2020.

\bibitem[IMZ23]{IMZ21}
Merlin Incerti-Medici and Abdul Zalloum.
\newblock Sublinearly morse boundaries from the viewpoint of combinatorics.
\newblock {\em Forum Mathematicum}, 35(4):1077--1103, 2023.

\bibitem[KM12]{KM12}
Jeremy Kahn and Vladimir Markovic.
\newblock Immersing almost geodesic surfaces in a closed hyperbolic three
  manifold.
\newblock {\em Annals of Mathematics}, pages 1127--1190, 2012.

\bibitem[MM99]{MM99}
Howard Masur and Yair Minsky.
\newblock Geometry of the complex of curves, i: Hyperbolicity.
\newblock {\em Invent. Math.}, 138:103--149, 1999.

\bibitem[MMS12]{MMS12}
Howard Masur, Lee Mosher, and Saul Schleimer.
\newblock {On train-track splitting sequences}.
\newblock {\em Duke Mathematical Journal}, 161(9):1613 -- 1656, 2012.

\bibitem[MQZ22]{MQZ20}
Devin Murray, Yulan Qing, and Abdul Zalloum.
\newblock Sublinearly morse geodesics in {CAT}(0) spaces: lower divergence and
  hyperplane characterization.
\newblock {\em Algebraic {\&} Geometric Topology}, 22(3):1337--1374, aug 2022.

\bibitem[NQ22]{NQ22}
Hoang~Thanh Nguyen and Yulan Qing.
\newblock Sublinearly morse boundary of cat (0) admissible groups.
\newblock {\em arXiv preprint arXiv:2203.00935}, 2022.

\bibitem[NR03]{NR03}
Graham~A Niblo and Lawrence~D Reeves.
\newblock Coxeter groups act on cat (0) cube complexes.
\newblock {\em Journal of Group Theory}, 6(3):399--413, 2003.

\bibitem[PSZ22]{PSZ22}
Harry Petyt, Davide Spriano, and Abdul Zalloum.
\newblock Hyperbolic models for cat (0) spaces.
\newblock {\em arXiv preprint arXiv:2207.14127}, 2022.

\bibitem[QR22]{QRT19}
Yulan Qing and Kasra Rafi.
\newblock Sublinearly morse boundary i: Cat (0) spaces.
\newblock {\em Advances in Mathematics}, 404:108442, 2022.

\bibitem[QRT20]{QRT20}
Yulan Qing, Kasra Rafi, and Giulio Tiozzo.
\newblock Sublinearly morse boundary ii: Proper geodesic spaces.
\newblock {\em arXiv preprint arXiv:2011.03481}, 2020.

\bibitem[Sag95]{Sag97}
Michah Sageev.
\newblock Ends of group pairs and non-positively curved cube complexes.
\newblock {\em Proceedings of the London Mathematical Society}, 3(3):585--617,
  1995.

\bibitem[She22]{She22}
Sam Shepherd.
\newblock A cubulation with no factor system.
\newblock {\em arXiv preprint arXiv:2208.10421}, 2022.

\bibitem[Wis04]{Wis04}
Daniel~T Wise.
\newblock Cubulating small cancellation groups.
\newblock {\em Geometric \& Functional Analysis GAFA}, 14(1):150--214, 2004.

\bibitem[Wis21]{Wis21}
Daniel~T Wise.
\newblock {\em The structure of groups with a quasiconvex hierarchy:(ams-209)}.
\newblock Princeton University Press, 2021.

\bibitem[Zal23]{Zal23}
Abdul Zalloum.
\newblock Injectivity, cubical approximations and equivariant wall structures
  beyond cat (0) cube complexes.
\newblock {\em arXiv preprint arXiv:2305.02951}, 2023.

\end{thebibliography}
\bibliographystyle{alpha}

\end{document}